\documentclass[12pt, reqno]{amsart}
\usepackage{hyperref}

\usepackage{amssymb,amsmath,graphicx,amsthm}
\usepackage{enumerate,xcolor}
\usepackage[normalem]{ulem}

\DeclareMathOperator{\supp}{supp}

\usepackage{color}

\def\sumJ(p-1){\underset{|J|=p-1}{{\sum}'}}

\def\sumKp-2{\underset{|K|=p-2}{{\sum}'}}

\def\sumjq+1{\underset {j\leq q+1}\sum}
\def\sumjn-1{\underset {j\leq n-1}\sum}

\def\bt{\begin{theorem}}
\def\el{\end{lemma}}
\def\bl{\begin{lemma}}
\def\et{\end{theorem}}
\def\bp{\begin{proposition}}
\def\ep{\end{proposition}}
\def\bd{\begin{definition}}
\def\ed{\end{definition}}
\def\br{\begin{remark}}
\def\er{\end{remark}}

\def\sumJ{\underset{|J|=k}{{\sum}'}}

\def\T{\text}

\newcommand{\om}{\omega}

\newcommand{\bom}{\bar{\omega}}

\newcommand{\we}{\wedge}

\newcommand{\no}[1]{\|{#1}\|}

\def\R{{\Bbb R}}

\def\I{{\mathcal I}}

\def\C{{\Bbb C}}

\def\Z{{\Bbb Z}}

\def\la{\langle}
\def\ra{\rangle}
\def\di{\partial}
\def\dib{\bar\partial}
\def\Label#1{\label{#1}{\bf(#1)}~}

\numberwithin{equation}{section}

\def\T{\text}

\def\Label#1{\label{#1}}

\theoremstyle{plain}

\newtheorem{theorem}{Theorem}[section]

\newtheorem{corollary}[theorem]{Corollary}

\newtheorem{lemma}[theorem]{Lemma}

\newtheorem{proposition}[theorem]{Proposition}

\theoremstyle{definition}

\newtheorem{definition}[theorem]{Definition}

\theoremstyle{remark}

\newtheorem{remark}[theorem]{Remark}

\DeclareMathOperator{\Rre}{Re}

\newcommand{\p}{\partial}

\newcommand{\dbar}{\bar\partial}
\newcommand{\dbars}{\bar\partial^*}

\newcommand{\dbarb}{\bar\partial_b}
\newcommand{\dbarbs}{\bar\partial^*_b}
\newcommand{\Boxb}{\Box_b}
\newcommand{\vp}{\varphi}

\newcommand{\atopp}[2]{\genfrac{}{}{0pt}{2}{#1}{#2}}
\newcommand{\nn}{\nonumber}
\newcommand{\eps}{\epsilon}
\newcommand{\opC}{\mathcal C}

\newcommand{\lam}{\lambda}
\DeclareMathOperator{\Tr}{Tr}

\newcommand{\opH}{\mathcal H}

\renewcommand{\H}{\mathcal H}

\def\L{\mathcal L }
\def\Re{\text{Re}}

\def\N{\mathcal N}
\begin{document}

\title{The Kohn-Laplace equation on abstract CR manifolds: Global regularity}         
      \author{Tran Vu Khanh and Andrew Raich}          
        \address{T.~V.~Khanh}
        \address{School of Mathematics and Applied Statistics, University of Wollongong, NSW 2522, Australia}
        \email{tkhanh@uow.edu.au}
        \address{A. Raich}
        \address{Department of Mathematical Sciences, SCEN 327, 1 University
        	of Arkansas, Fayetteville, AR 72701, USA} 
        \email{araich@uark.edu}
        \thanks{The first author was supported by ARC grant
        	DE160100173. The second author was partially supported by NSF grant DMS-1405100. This work was done in part while the authors were  visiting members at the Vietnam Institute for Advanced Study in Mathematics (VIASM). They would like to thank the institution for its hospitality and support.}
	\thanks{}

\maketitle

\begin{abstract}Let  $M$ be a compact, pseudoconvex-oriented, $(2n+1)$-dimensional, abstract CR manifold of hypersurface type, $n\geq 2$. 
We prove the following:
\begin{enumerate}[(i)]
 \item If $M$ admits a strictly CR-plurisubharmonic function on $(0,q_0)$-forms, then the complex Green operator $G_q$ exists and  is continuous on $L^2_{0,q}(M)$ for degrees $q_0\le q\le n-q_0$. In the case
 that $q_0=1$, we also establish continuity for $G_0$  and $G_n$. Additionally, 
 the $\dbarb$-equation on $M$ can be solved in $C^\infty(M)$.
 \item If $M$ satisfies ``a weak compactness property"  on $(0,q_0)$-forms, then $G_q$ is a continuous operator on $H^s_{0,q}(M)$ and is therefore globally regular on $M$ for degrees $q_0\le q\le n-q_0$;  
 and also for the top degrees  $q=0$ and $q=n$ in the case $q_0=1$. 
\end{enumerate}
We also introduce the notion of a ``plurisubharmonic CR manifold" and show that it generalizes the notion of ``plurisubharmonic defining function" for a a domain in $\C^N$ and implies
that $M$ satisfies the weak compactness property.
\end{abstract} 

\section{Introduction}
\Label{s1}

Let $M$ be an abstract compact smooth CR manifold of real dimension $2n+1$ equipped with a Cauchy-Riemann structure $T^{1,0}M$. The tangential Cauchy-Riemann operator $\dbarb$ is well-defined on smooth $(p,q)$-forms
and our interest is to understand the regularity of the canonical solution $u$ to the $\dbarb$-equation, $\dbarb u =\vp$, 
i.e., the solution of the $\dbarb$-equation that is orthogonal to the kernel of $\dbar_b$. Our approach is via $L^2$-methods and we obtain the canonical solution to the $\dbarb$-equation by solving
the related $\Boxb$-equation $\Boxb u=f$ where the Kohn Laplacian $\Boxb = \dbarb\dbarbs+\dbarbs\dbarb$ and $\dbarbs$ is the $L^2$-adjoint of $\dbarb$ in an appropriate $L^2$-space. We provide the technical details
below. The operator $\Boxb$ maps $(0,q)$-forms to $(0,q)$-forms and its inverse on $(0,q)$-forms (when it exists) is called the \emph{complex Green operator} and denote $G_q$. Our primary goal in this paper to 
find sufficient conditions for the global regularity and exact regularity of $G_q$ and related operators. Global regularity means that $G_q$ maps smooth forms to smooth forms, and exact regularity means that
$G_q$ is continuous on all of the $L^2$-Sobolev spaces $H^s_{0,q}(M)$, $s\geq 0$. Before we can investigate the exact and global regularity questions, we must generalize the existing $L^2$-theory for solving
the $\Boxb$ and $\dbarb$-equations in $L^2_{0,q}(M)$. Additionally, we define a family of weighted Kohn Laplacians and show that for a given index $s \geq 0$, there are weighted complex Green operators whose
inverses are continuous on $H^s_{0,q}(M)$.

\subsection{History -- the $L^2$-theory of $\dbarb$ on CR manifolds}
A classical result of H\"ormander is that solvability of an operator (e.g., $\dbarb$, $\Boxb$, etc.) is equivalent to it having closed range. Solvability is closely linked to geometric and potential theoretic conditions
on the manifold $M$. Curvature, for example, is measured by the Levi form. The Levi form on boundaries of domains in $\C^{n+1}$ is essentially the complex analysis analog of the second fundamental form, and
when it is nonnegative, $M$ is called pseudoconvex.
The first results establishing closed range of $\dbarb$ in $L^2_{0,q}(M)$ occur when $M\subset\C^{n+1}$, $n\ge 1$, is pseudoconvex. 
Shaw \cite{Sha85} proves that $\dbarb$ has closed range  for
$0 \leq q \leq n-1$ when $M$ is the boundary of a pseudoconvex domain. With Boas, Shaw supplements her earlier result and establishes the result in the top degree $q=n$ \cite{BoSh86}. 
Using a microlocal argument, and in fact, developing the microlocal machinery, 
Kohn proves that if $M$ is compact, pseudoconvex, and the boundary of a smooth complex manifold which admits a strictly plurisubharmonic function defined in a neighborhood of $M$, 
then $\dib_b$ has closed range \cite{Koh86}.  In \cite{Nic06} Nicoara extends the Kohn microlocal machine and proves that on an embedded, compact, pseudoconvex-oriented CR manifold of dimension $(2n+1)\geq5$, 
$\dbarb$ has closed range and can be solved in $C^\infty$.

When pseudoconvexity is relaxed, Harrington and Raich prove closed range and related estimates at the level of $(0,q)$-forms for a fixed $q$, $1 \leq q \leq n-1$,  
by developing a weak $Y(q)$ condition that is much weaker than pseudoconvexity but still suffices \cite{HaRa11, HaRa15}. They prove
similar results to those of Nicoara \cite{Nic06}. Harrington
and Raich work on embedded manifolds in \cite{HaRa11} and in a Stein manifold in \cite{HaRa15}. In \cite{HaRa15}, they adopt Shaw's techniques to work on manifolds of minimal smoothness.
In all of the above cases, 
a crucial property of the manifolds is the existence a strictly CR plurisubharmonic function in a neighborhood of $M$. Indeed, our hypotheses in Theorem \ref{thm:L^2-theory} below include the existence of such a function, and 
our main task in proving Theorem \ref{thm:L^2-theory} is to show that the arguments from the previous works hold in the generality in which we work. The main focus of this paper, however, is the global and exact regularity statements.

\subsection{History -- global regularity for $G_q$}
An operator is \emph{globally regular} if it preserves $C^\infty$. Determining necessary and sufficient conditions for the global regularity of the complex Green operator (and $\dbar$-Neumann operator) is a one of the major
questions in the $L^2$-theory of the tangential Cauchy-Riemann operators.  
On compact manifolds, it is clear that if $G_q$ preserve $C^\infty$ locally, then it will preserve $C^\infty$ globally. However,
local regularity of $G_q$ requires  strong hypotheses, e.g., subelliptic estimates \cite{Koh85}, 
subelliptic estimates with multipliers \cite{Koh00, BaPiZa15}, or superlogarithmic estimates \cite{Koh02, KhZa11}.  
In parallel with this work \cite{Kha18}, the first author introduces a new potential theoretical condition named the $\sigma$-superlogarithmic property 
to prove the local regularity of $G_q$. However, this condition is far stronger than necessary to imply global regularity.

The techniques to prove global regularity do 
so by showing that $G_q$ is \emph{exactly regular}, meaning that
$G_q$ is a bounded operator from $H^s_{0,q}(M)$ to itself. 
There are several known conditions that suffice to prove exact (and hence global) regularity. 
The first is that  $M$ is the boundary of a domain that admits a plurisubharmonic
defining function \cite{BoSt91}. 
The second is when $G_q$ is a compact operator.   $G_q$ is known to be a compact operator by Raich \cite{Rai10} when 
$M$ is a smooth, orientable, pseudoconvex CR manifold of hypersurface type of real dimension at least five that satisfies a pair of potential-theoretical conditions that he names properties  $(CR\T-P_q)$ and $(CR\T-P_{n-1-q})$.  
Straube \cite{Str12} improves Raich's result by shedding the orientability hypothesis and showing that $(CR\T-P_q)$ is equivalent to $(P_q)$ (see \cite{Str08} for details and background on $(P_q)$). 
Pinton, Zampieri, and Khanh further reduce the hypotheses by removing the embeddability requirement so that
$M$ is an abstract CR manifold \cite{KhPiZa12a}. They also
prove a (real) dimension 3 compactness result under the additional assumption that $\dbarb$ has closed range on functions.
The third known condition for global and exact regularity of $G_q$ is the existence of special 1-form that is exact 
on the null-space of the Levi form. Straube and Zeytuncu prove global regularity for $G_q$ on a smooth, oriented, compact, pseudoconvex CR submanifold in $\C^n$ under this hypothesis \cite{StZe15}.

The genesis of the three known sufficient conditions is that they, or their analogs, are sufficient conditions for the global/exact regularity of the $\dbar$-Neumann operator on domains
in $\C^n$. See  Boas and Straube \cite{BoSt99} for a discussion of the early history of the global regularity question and Straube \cite{Str08} or Harrington \cite{Har11} for recent work.

\subsection{Main Goal and Major Results}
The main goal of this paper is to prove a global regularity result which encompasses and improves on all of the existing work. In particular, our condition is weaker than Straube and Zeytuncu's in two respects: 1) we do not require
exactness of the special 1-form $\alpha$, only a bound in the spirit of \cite{Str08}, and 2) we no longer require $M$ to be embedded. We state the main result below but defer the technical definitions to 
Section \ref{sec:definitions}, though we do need the form $\gamma$ that is dual to the ``bad" direction $T$, that is, the dual to the vector that is orthogonal to the CR structure of the manifold. 
Also $d\gamma$ is connected to the Levi form and $\L_{\lambda}$ is the analog of the complex Hessian for the function $\lambda$.
 \bt \label{thm:global regularity}
Let $M$ be a smooth, compact, pseudoconvex-oriented CR manifold of dimension $2n+1$ which admits a CR plurisubharmonic function on $(0,q_0)$-forms, $n\geq 2$. 
 Assume that for every $\eps>0$ there exist a $C^\infty$ real valued function $\lambda_\eps$, a purely imaginary vector field $T_\eps$, and a constant $A_\eps>0$  so that $|\lambda_\eps|$,
$\gamma(T_\eps)$ are uniformly bounded, $\gamma(T_\eps)$ is bounded away from zero, and
\begin{equation}\label{eqn:1/eps est}
\la(\L_{\lambda_\eps}+A_\eps d\gamma)\lrcorner u,\bar u\ra\ge \frac{1}{\eps}|\alpha_\eps|^2|u|^2
\end{equation}
for any $(0,q_0)$-forms $u$. The form $\alpha_\eps$ is real and defined by $\alpha_\eps=-\T{\{Lie\}}_{T_\eps}(\gamma).$

If $q_0\le q\le n-q_0$,   then 
the operators $G_q$, $\dib_bG_q$, $G_q\dib_b$, $\dib_b^*G_q$, $G_q\dib_b^*$, $I-\dib_b^*\dib_bG_q$, $I-\dib_b^*G_q\dib_b$, $I-\dib_b^*\dib_bG_q$, 
$I-\dib_bG_q\dib_b^*$, $\dib_b^*G^2_q\dib_b$ and $\dib_bG_q^2\dib_b^*$ are both globally regular and exactly regular in the $L^2$-Sobolev space $H^s$, $s\geq 0$. 
In the case $q=1$, the operators $G_0 = \dbarbs G_1^2 \dbarb$ and $G_n = \dbarb G_{n-1}^2 \dbarbs$ 
and hence $G_0$, $\dib_bG_0$, $G_n$, $\dib_b^*G_n$ are both globally regular and exactly regular in $H^s$ as well, $s\geq 0$.
\et
\begin{remark}
The pseudoconvex-oriented is a necessary condition for working on global abstract CR manifolds containing open Levi flat sets in the sense that on such a Levi flat set we can choose local contact forms 
whose Levi forms are locally pseudoconvex at every point at which they are defined, but there does not exist a global contact form whose Levi form is globally pseudoconvex.
\end{remark}

   The hypothesis of Theorem~\ref{thm:global regularity} is weaker than property $(CR\T-P_q)$ (introduced by Raich \cite{Rai10}).
   Furthermore, if $\alpha=:-\T{\{Lie\}}_{T}(\gamma)$  is exact on the null-space of the Levi form then  Straube and Zeytuncu \cite[Proposition 1]{StZe15} prove that for each $\eps>0$ 
   there exists  $T_\eps$  such that $\alpha_\eps = -\T{\{Lie\}}_{T_\eps}(\gamma)$ satisfies 
   $|\alpha_\eps|\le\eps$. So if we choose $\lambda_\eps:=\lambda$ is a strictly CR-plurisubharmonic with associated constant $A_\eps=1$ as  in \eqref{CRpluri} below,  then \eqref{eqn:1/eps est} is satisfied for all
   $\eps>0$, and hence the hypotheses of Theorem~\ref{thm:global regularity} holds. The following two corollaries are proven using the arguments of \cite{StZe15}.
 \begin{corollary}\label{cor:CRpsh on (0,1)}
    	Let $M$ be a smooth, compact, pseudoconvex-oriented CR manifold of dimension $2n+1$ which admits a CR plurisubharmonic function on $(0,1)$-forms, $n\geq 2$. 
	If the real 1-form $\alpha=-\{\T{Lie}\}_{T}(\gamma)$ is exact on the null space of the Levi form, then the hypothesis in Theorem~\ref{thm:global regularity} holds for $(0,1)$-forms.
    	\end{corollary}
 \begin{corollary}	Let $M$ be a smooth, compact, pseudoconvex-oriented CR manifold of dimension $2n+1$ which admits a CR plurisubharmonic function on $(0,1)$-forms, $n\geq 2$. 
   	Let $S$ be the set of non-strictly pseudoconvex points. Suppose that 
   	at each point of $S$, the (real) tangent space is contained in the null space of
   	the Levi form at the point. If the first de Rham cohomology of $S$ is
   	trivial, then $\alpha$ is exact on the null space of the Levi form (this happens if $S$ is simply connected). Consequently, the
   	hypotheses of Corollary \ref{cor:CRpsh on (0,1)} hold.
   \end{corollary}  
   \begin{proof} The proof follows exactly from \cite[Theorem 3]{StZe15}. Let $\N_x$ be the null space of the Levi form at $x$.
   	 Analogously to Lemma~\cite[Lemma on p. 230]{BoSt93}, 
	 we have $(d\alpha|_x)(X\we Y)=0$ if $X,Y\in \N_x\oplus \overline{\N_x}$. Thus, $d\alpha=0$ on $S$. 
	 Since the first DeRham cohomology of $S$ is trivial,  there exists $h$ on $C^\infty(\overline{S})$ such that such that $\alpha=dh$ on $S$. This mean, $\alpha$ is exact on $\N_x$. 
   	   \end{proof}

   Also in \cite{StZe15}, the authors show that if the defining functions of $M$ are plurisubharmonic in some neighborhood of $M$ 
   then $\alpha$ is exact on the null space of the Levi form. Their proof strongly relies on the fact that $M$ is embedded in $\C^n$, however, we are able to remove the embeddability assumption.
  \bt\label{thm:CRpsh (0,1)}
    Let $M$ be a smooth, compact, plurisubharmonic-oriented CR manifold of dimension $2n+1$ that admits a CR plurisubharmonic function on $(0,1)$-forms. 
    Then for any $\eps>0$ there exists a vector $T_\eps$ whose length is bounded and bounded away from $0$ with associated form $\alpha_\eps = -\T{\{Lie\}}_{T_\eps}(\gamma)$ satisfying
    $|\alpha_\eps|<\eps$. Consequently, the conclusion of Theorem \ref{thm:global regularity} holds for all $0 \leq q \leq n$.
    \et

The secondary goal of this paper is to restate in a slightly more general form the $L^2$-theory for $\dbarb$ without the embeddability assumption.
%
%
\begin{theorem}\label{thm:L^2-theory}
Let $M$ be a smooth, compact, pseudoconvex-oriented CR manifold of dimension $2n+1$ that admits a strictly CR-plurisubharmonic function on $(0,q_0)$-forms. If $q_0 \leq q \leq n-q_0$ then the following hold:
\begin{enumerate}\renewcommand{\labelenumi}{(\roman{enumi})}
	\item The $L^2$ basic estimate 
	\begin{equation}
	\label{L2basic} \no{u}_{L^2}^2\le c(\no{\dib_bu}_{L^2}^2+\no{\dib_b^*u}_{L^2}^2)
	\end{equation}
	holds for all $u\in \T{Dom}(\dib_b)\cap\T{Dom}(\dib_b)\cap \H^{\perp}_{0,q}(M)$.
\item The operators
$\dbarb: L^2_{0,\tilde q}(M)\to L^2_{0,\tilde q+1}(M)$
and
$\dbarbs: L^2_{0,\tilde q+1}(M)\to L^2_{0,\tilde q}(M)$
have closed range when $\tilde q = q$ or $q-1$.
Additionally, $\Box_b:L^2_{0,q}(M)\to L^2_{0,q}(M)$ has closed range;
\item The operators $G_q$, $\dbarbs G_q$, $G_q\dbarbs$, $\dbarb G_q$, $G_q\dbarb$, $I - \dbarbs \dbarb G_q$, $I - \dbarbs G_q \dbarb$, $I - \dbarb \dbarbs G_q$, $I - \dbarb G_q \dbarbs$ are $L^2$ bounded. In the case $q=1$, the operators $G_0 = \dbarbs G_1^2 \dbarb$ and $G_n = \dbarb G_{n-1}^2 \dbarbs$ and hence $G_0$, $\dib_bG_0$, $G_n$, $\dib_b^*G_n$ are continuous on  $L^2$. 
\item The $\dib_b$-equation $\dib_bu=\varphi$ has a solution $u\in C^\infty_{0,\tilde q-1}(M)$ for $\tilde q=q$ or $q+1$ if $\varphi$ is a $\dib_b$-closed, $C^\infty$-smooth $(0,\tilde q)$-form.
\item The space of harmonic forms $\mathcal H_{0,q}(M)$ is finite dimensional.
\end{enumerate}
\end{theorem}
\begin{remark} There is little new work to be done to prove Theorem \ref{thm:L^2-theory}. Our hypotheses are exactly what Nicoara \cite{Nic06} and Harrington and Raich \cite{HaRa11} use. 
In their work, embeddability is only used to establish 
the existence of a strictly CR plurisubharmonic function. We will simply highlight aspects of the earlier proofs that we need to prove Theorem \ref{thm:L^2-theory}.

Our main contribution is the development of an elliptic regularization that does not require $M$ to be embedded, whereas the previous methods strongly used the embeddedness.
\end{remark}

The outline of the rest of the paper is as follows. The technical preliminaries are given in Section \ref{sec:definitions}. The establishment of a basic estimate comprises the beginning of
Section \ref{sec:basic estimate}. The remainder of the section develops the microlocal framework that we use to prove the main theorem. We prove Theorem \ref{thm:L^2-theory} in Section \ref{sec:proving L^2 theory}. 
Its proof follows the argument of \cite{HaRa11} (and in \cite{Nic06}), and like \cite{HaRa11}, the microlocal argument
proves an auxiliary result on a carefully constructed weighted spaces, Theorem \ref{thm:main theorem for weighted spaces}. It is through the weighted spaces that we can solve $\dbarb$ in
$C^\infty$. We conclude the paper in Section \ref{sec:hypoellipticity} with proofs of Theorem \ref{thm:L^2-theory}
and Theorem \ref{thm:CRpsh (0,1)}.

%
%
\section{Definitions}\label{sec:definitions}
Let $M$ be a real smooth manifold of dimension $2n+1$, $n\ge 1$. Let $\C TM$ be the complexified tangent bundle over $M$, and $T^{1,0}M$ be a subbundle of $\C TM$. We say that $M$ is a 
\emph{CR manifold of hypersurface type} equipped with CR structure $T^{1,0}M$ (or CR manifold for short) if the following conditions are satisfied:
\begin{enumerate}
	\item[(i)] $\dim_\C T^{1,0}M=n$,
	\item[(ii)] $T^{1,0}M\cap T^{0,1}M=\{0\}$, where $T^{0,1}M=\overline{T^{1,0}M}$,
	\item[(iii)] for any $L, L'\in \Gamma(U, T^{1,0}M)$, the Lie bracket $[L,L']$ is still in $\Gamma(U, T^{1,0}M)$, where $U$ is any open set of $M$ and $\Gamma(U, T^{1,0}M)$ denotes the space of 
	smooth sections of $T^{1,0}M$ over $U$ (this condition is nonexistent when $n=1$).
\end{enumerate} 
On $M$, we choose a Riemann metric $\la\cdot,\cdot\ra$ which induces a Hermitian metric on $T^{1,0}M\oplus T^{0,1}M$
so that $\la L,\bar L'\ra=0$ for any $L,L'\in T^{1,0}M$.  Given the metric $\la\cdot,\cdot\ra$, we choose a local frame $\{L_1,\dots, L_n\}$ of $T^{1,0}M$ and a purely
imaginary vector field $T$ that is orthogonal (and hence transversal) to 
$T^{1,0}M\oplus T^{0,1}M$ so that $\{L_1,\dots,L_n,\bar L_1,\dots,\bar L_n, T\}$ forms an orthonormal basis of
$\C TM$. See \cite{Bog91} for details. 
Denote by $\omega_1,...,\omega_{n},\bar \om_1,...,\bar\om_{n},\gamma$ the dual basis of $1$-forms for $L_1,...,L_{n}$, $\bar L_1,...,\bar L_{n}$, and $T$. 
We call any open set $U \subset M$  a {\it local patch} if it admits such vectors and forms. 
Our interest is in global solvability, so we need a suitable partition of unity. We call a cover $\{U_\mu\}$ of $M$ a \emph{good cover} if each $U_\alpha$ is a local patch. We also let $\{\eta_\mu\}$ be a partition of
unity subordinate to $\{U_\mu\}$.\\

For a $C^2$ function $\phi$ on $M$, we call the alternating $(1,1)$-form $\L_\phi=\frac{1}{2}\big(\di_b\dib_b-\dib_b\di_b\big)\phi$ on $T^{1,0}M\times T^{0,1}M$ the 
\emph{Levi form of $\phi$}.
The Levi form $\mathcal L_M$ of $M$ at $x$ is the Hermitian form given by $d\gamma (L_x \wedge \bar L_x') = \gamma([\bar L_x',L_x])$  for $L,\,L'\in T^{1,0}_x$. We say that
$M$ is \emph{pseudoconvex at $x\in M$} if the Levi form is positive semidefinite in a neighborhood of $x$, i.e., $d\gamma(L\wedge\bar L)\geq 0$ for all $L\in T^{1,0}_x$. 
$M$ is \emph{orientable} if there exists is a global 1-form $\tilde \gamma$. We say that $M$ is \emph{pseudoconvex-oriented} if the $1$-form $\gamma \in \C TM^*$ is globally defined and 
the Levi form is positive semidefinite for all $x\in M$.
All of the manifolds that we consider in this paper will be pseudoconvex-oriented. 
We express $d\gamma$ in a local patch $U$ as 
\begin{equation}\label{dgamma}
	d\gamma=-\left(\sum_{i,j=1}^nc_{ij}\om_i\we\bom_j+\sum_{j=1}^nc_{0j}\gamma\we \om_j+\sum_{j=1}^n\bar c_{0j}\gamma\we \bom_j\right),\end{equation}
where the integrability condition of CR structure and the Cartan formula forces the coefficients of $\om_i\we \om_j$ and $\bom_j\we\bom_j$ to be zero. 
The pseudoconvex-oriented condition is equivalent to the Levi matrix  $\L_{M}:=\{c_{ij}\}_{ij=1}^n$ being positive semidefinite at every $x\in M$. 
We extend the $n\times n$ Levi matrix  to a $(n+1)\times (n+1)$ matrix $\{c_{ij}\}_{ij=0}^n$ 
with entries $c_{0j}$ for $j=1,\dots,n$, $c_{j0}=\overline{c_{0j}}$, and $c_{00}$ to be chosen. We say that $M$ is \emph{plurisubharmonic} at $x$ 
if there exists $c_{00}$ such that the extended Levi matrix $\{c_{ij}\}_{ij=0}^n\ge 0$ in a neighborhood of $x$ 
and $M$ is {\it plurisubharmonic-oriented}  if it is pseudoconvex-oriented and plurisubharmonic at every $x\in M$. 
It is obvious that if $M$ is embedded in a Stein manifold $X$ and admits a plurisubharmonic defining function $r$ then $M$ is plurisubharmonic-oriented. Indeed, in this case the plurisubharmonic-oriented condition is fulfilled 
if we choose  $\gamma=\frac{i}{2}(\di r-\dib r)$ and $c_{00}$ is the $T:=L_0-\bar L_0$ component of $[L_0,\bar L_0]$ where $L_0\in T^{1,0}X$ is the dual of $\di r$.   
Let $\alpha=\sum_{j=1}^n \left(c_{0j}\om_j+ \bar c_{0j}\bom_j\right)$ and observe that $\alpha = -\{\T{Lie}\}_T(\gamma)$. In \cite{StZe15}, the $(1,0)$-form $\alpha$ is called {\it exact on
the null space of the Levi form} if there exists a smooth function $h$, defined in
a neighborhood of $K$, the set of weakly pseudoconvex points of $M$,
such that
$$dh(L_z)(z) = \alpha(L_z)(z),\quad L_z\in \mathcal N_z\cap T^{1,0}_z, \quad z\in K,$$
where $\mathcal N_z$ is the null space of the Levi form at $z\in K$. Below, we will show that if $M$ is plurisubharmonic-oriented then $\alpha$ is exact on the null space of the Levi form.\\

Denote by $dV$ the element of volume on $M$, the induced $L_2$-inner product and norm on $ C^\infty_{p,q}(M)$ is defined by
$$
(u,v)=\int_M\la u, \bar v\ra dV,\qquad \no{u}_{L^2}^2=(u,u).
$$
The function space $L_{p,q}^2(M)$  is the Hilbert space obtained by completing $ C^\infty_{p,q}(M)$ under the $L^2$-norm. The Sobolev spaces $H^s_{p,q}(M)$ are obtained by completing
$ C^\infty_{p,q}(M)$ under the usual $H^s(M)$ norm,  $\no{\cdot}_{H^s}$, applied componentwise.
We now want to define $\dbarb$ on $(p,q)$-forms, extending our definition from functions.
Define the operator $\dib_b:  C^\infty_{p,q}(M)\to  C^\infty_{p,q+1}(M)$ to be the projection of the de Rham exterior differential operator $d$ to $ C^\infty_{p,q}(M)$. 
We denote by $\dib^*_b:\, C^\infty_{p,q+1}(M)\to  C^\infty_{p,q}(M)$ the $L^2$-adjoint of $\dib_b$ and define the Kohn-Laplacian by 
$$
\Box_b:=\dib_b\dib_b^*+\dib_b^*\dib_b: C^\infty_{p,q}(M) \to C^\infty_{p,q}(M).
$$ 
The space of harmonic $(p,q)$-forms $\mathcal H_{p,q}(M):=\ker\Box_{b}$ coincides with $\ker \dib_b\cap\ker\dib^*_b$. 
Our  result will be to find a sufficient condition so that the estimate
\begin{equation}
\label{1.1}
\no{u}_{L^2}\leq c(\no{\dib_b u}_{L^2}+\no{\dib^*_bu}_{L^2})\quad \T{for }u\in \T{Dom}(\dib_b)\cap Dom(\dib^*_b)\cap \mathcal H^\perp_{p,q},
\end{equation}
holds and $\mathcal H_{p,q}(M)$ is finite dimensional. 
A consequence of this estimate is that $\Box_{b}$  has a  bounded inverse on $\H^\perp_{p,q}(M)$. In this case, we extend the inverse to be identically $0$ on $\H_{p,q}(M)$.
\bd
In the case that $\Boxb$ is an invertible operator on $L^2_{p,q}(M)\cap\H^\perp_{p,q}(M)$, we denote the inverse by $G_{p,q}$ and call it the \emph{complex Green operator}.
\ed
The existence of a strictly CR plurisubharmonic function means that the curvature of $M$ does not play a factor in the existence of $G_{p,q}$, so it suffices to take $p=0$, and we denote $G_{0,q}$ by
$G_q$.

If $M$ is not embedded in a Stein manifold, we can define strictly CR plurisubharmonic functions as follows. 
\begin{definition} Let $M$ be a pseudoconvex CR manifold. A $C^\infty$ real-valued function $\lambda$ defined on $M$ is \emph{strictly CR plurisubharmonic on $(0,q)$-forms} if 
there exists a constant $a>0$ so that 
so that 
\begin{eqnarray}
\label{CRpluri}
\la \left(\L_\lambda+ d\gamma\right)\lrcorner u,\bar u\ra\ge a|u|^2,
\end{eqnarray}
for any $u\in  C^\infty_{0,q}(M)$. 
\end{definition}
For $u$ defined on $U$, the contraction operator $\lrcorner$ is defined by 
$$
\theta\lrcorner u=\sum_{I\in\I_{q-1}} \left(\sum_{j=1}^n \theta_j u_{jI}\right)\bom_I\quad\T{if $\theta=\sum_j\theta_j\,\om_j$ is a $(1,0)$-form on $U$};
$$ 
and
$$\quad\theta\lrcorner u=\sum_{I\in\I_{q-1}}\sum_{j=1}^n \left(\sum_{i=1}^n \theta_{ij} u_{iI}\right)\bom_j\we\bom_{I}\quad\T{if $\theta=\sum_{i,j=1}^n \theta_{ij}\,\om_i\we\bom_j$ is a $(1,1)$-form on $U$}.$$
Thus, the Levi form $d\gamma$ and $\L_\phi$ acting on $(0,q)$-forms $u,v$ defined in $U$ can be expressed as
$$\la d\gamma\lrcorner u,\bar v\ra=\sum_{I\in\I_{q-1}}\sum_{i,j=1}^nc_{ij}u_{iI}\overline{v_{jI}}\quad \T{and } \la \L_\phi\lrcorner u,\bar v\ra=\sum_{I\in\I_{q-1}}\sum_{i,j=1}^n\phi_{ij}u_{iI}\overline{v_{jI}}.$$

\begin{remark}
 	Strictly CR-plurisubharmonic functions always exist if $M$ is strictly pseudoconvex or embedded into a Stein manifold. They do not, however, always exist on abstract CR manifolds. 
	See, for example, Grauert's example \cite{Gra63}. 
 \end{remark}

\subsection{Working in local coordinates}\label{subsec:prelims for the basic est}
Let 
 $U$ be a local patch of $M$ with its associated basis of tangential vector fields
$L_1,...,L_{n},\bar L_1,...,\bar L_{n},T$ and dual basis
$\omega_1,...,\omega_{n},\bar\omega_1,...,\bar\omega_{n}, \gamma$. 
For the moment, we work locally on $U$. 

The  condition of pseudoconvexity of $M$ at $x$ is equivalent to the Levi matrix $\{c_{ij}\}_{i,j=1}^n\ge 0$ in a neighborhood of $x$. 
We recall that $M$ is {\it pseudoconvex-oriented} (resp., {\it plurisubharmonic-oriented}) if there exist a global 1-form section $\gamma$ 
(resp., a global 1-form section $\gamma$ and a smooth function $c_{00}$) such that the Levi matrix $\{c_{ij}\}_{ij=1}^n$ (resp. the extended Levi matrix $\{c_{ij}\}_{ij=0}^n$) is positive semidefinite for every $x\in M$.

We further define $c_{ij}^k$ to be the $L_k$-component of $[L_i,\bar L_j]$. Since $d$ applied to a $(1,0)$-form can produce a
$(2,0)$-form and a $(1,1)$-form \cite[\S8.2, Lemma 1]{Bog91}, it follows from the definition of $\dbarb$ that
\begin{equation}
\Label{31}
\begin{split}
c_{ij}^k:= \om_k([L_i,\bar L_j]) \underset{\T{Cartan}}=- \dbarb\om_k(L_i\wedge\bar L_j).
\end{split}
\end{equation}
Therefore,
\begin{equation}
\Label{32}
\dbarb\om_k=-\sum_{i,j=1}^nc_{ij}^k\, \om_i\wedge\bom_j.
\end{equation}
and conjugating yields
\begin{equation}
\Label{33}
\p_b\bom_k=\sum_{i,j=1}^n\overline{c_{ji}^k}\, \om_i\wedge\bom_j.
\end{equation}
Using Cartan's formula again, we conclude that $-\overline{c_{ji}^k}$ coincides with the $\bar L_k$-component of $[L_i,\bar L_j]$. Thus the full commutator is expressed by
\begin{equation}
\Label{34}
[L_i,\bar L_j]=c_{ij}T+\sum_{k=1}^n c_{ij}^kL_k - \sum_{k=1}^n\overline{ c_{ji}^k}\bar L_k.
\end{equation}
For a smooth function $\phi$ on $U$, we want to describe the matrix $(\phi_{ij})$ of the Hermitian form $\frac{1}{2}\big(\di_b\dib_b-\dib_b\di_b\big)\phi$. Now, $\dib_b\phi=\sum_{k=1}^n\bar L_k(\phi)\bom_k$ and therefore
\begin{equation}
\Label{35}
\begin{split}
\di_b\dib_b\phi&=\di_b\Big(\sum_{k=1}^n\bar L_k(\phi)\bom_k\Big)
\\&\underset{\T{by \eqref{33}}}=\sum_{i,j=1}^n\left(L_i\bar L_j(\phi)+\sum_{k=1}^n\overline{c_{ji}^k}\bar L_k(\phi)\right)\om_i\wedge\bom_j.
\end{split}
\end{equation}
Similarly,
\begin{equation}
\Label{36}
\begin{split}
\dib_b\di_b\phi&=\dib_b(\sum_{k=1}^nL_k(\phi)\om_k)
\\
&\underset{\T{by \eqref{32}}}=\sum_{i,j=1}^n \left(-\bar L_jL_i(\phi)-\sum_{k=1}^nc_{ij}^kL_k(\phi)\right)\om_i\wedge \bom_j.
\end{split}
\end{equation}
Combining \eqref{35} with \eqref{36} we get
\begin{equation}
\Label{37}
\begin{split}
\phi_{ij}&= \frac{1}{2}\Big(\di_b\dib_b\phi-\dib_b\di_b\phi\Big)\big(L_i\wedge\bar L_j\big)
\\
&=\frac12\left(L_i\bar L_j(\phi)+\bar L_jL_i(\phi)+\sum_{k=1}^n\left(\overline{c_{ji}^k}\bar L_k(\phi)+c_{ij}^kL_k(\phi)\right)\right)
\\
&=\bar L_jL_i(\phi)+\frac12\left([L_i,\bar L_j](\phi)+\sum_{k=1}^n\left(\overline{c_{ji}^k}\bar L_k(\phi)+c_{ij}^kL_k(\phi)\right)\right)
\\
&\underset{\T{by \eqref{34}}}=\bar L_jL_i(\phi)+\frac12c_{ij}T(\phi)+\sum_{k=1}^nc_{ij}^kL_k(\phi).
\end{split}
\end{equation}

To express a form in local coordinates,  let
$\I_q = \{(j_1,\dots,j_q) \in \N^q : 1 \leq j_1 < \cdots < j_q \leq n\}$, and for $J\in \I_q$, $I\in\I_{q-1}$, and $j\in\mathbb N$, $\eps^{jI}_J$ be the sign of the permutation $\{j,I\}\to J$
if $\{j\} \cup I = J$ as sets, and $0$ otherwise.
If $u\in   C^\infty_{0,q}(M)$, then $u$ is expressed locally as a combination
$$
u= \sum_{J\in\I_q} u_J\, \bar\omega_J,
$$
of basis forms $\bar\omega_J=\bar\omega_1\wedge...\wedge\bar\omega_{j_q}$ where $J = (j_1,\dots,j_q)$ and $C^\infty$-coefficients $u_J$.

We can also express the operator $\dib_b :  C^\infty_{0,q}(M)\to  C^\infty_{0,q+1}(M)$ and its $L_2$ adjoint $\dib^*_b :  C^\infty_{0,q+1}(M)\to  C^\infty_{0,q}(M)$ in the local basis as follows:
\begin{align}\label{dib}
\dib_b u&=  \sum_{\atopp{J\in\I_q}{K\in\I_{q+1}}} \sum_{k=1}^n \eps^{kJ}_K \bar L_k u_J\, \bom_K + \sum_{\atopp{J\in\I_q}{K\in\I_{q+1}}} b_{JK} u_J\, \bom_K 
\end{align}
and
\begin{eqnarray}\Label{dib*}
\dib_b^*v=- \sum_{J\in\I_q}\left(\sum_{j=1}^n L_j v_{jK}+ \sum_{K\in\I_{q+1}}a_{JK} v_K\right)\bom_J
\end{eqnarray}
where $b_{JK}, a_{JK}\in C^\infty(U)$.\\

%
%
\section{The basic estimate on CR manifolds}\label{sec:basic estimate}
In this section, we will work with the weighted $L_\phi^2$-norm  defined by
$$\no{u}_{L^2_\phi}^2=(u,u)_\phi:=\no{ue^{-\frac{\phi}{2}}}_{L^2}^2=\int_M\la u,\bar u\ra e^{-\phi}\,dV.$$
Let $\dib^*_{b,\phi}$ be the $L_\phi^2$-adjoint  of $\dib_b$. It is easy to see that for forms $u\in C^\infty_{0,q+1}(M)$ supported on $U_\mu$
\begin{equation}\Label{dib*phi}
\dib_{b,\phi}^*u=- \sum_{J\in\I_q}\left(\sum_{j=1}^n \sum_{K\in\I_{q+1}} \delta_j^\phi u_{jK}+ \sum_{K\in\I_{q+1}}a_{JK}u_K\right)\bom_J
\end{equation}
where $\delta^\phi_j \varphi:=e^\phi L_j(e^{-\phi}\varphi)$ and $a_{JK}\in C^\infty(U)$. For such $u$, 
$$
\di_b(\phi)\lrcorner u= - [\dib_b^*,\phi]u=[\dbarb,\phi]^*u=\sum_{I \in \I_{q-1}} \sum_{j=1}^n L_j(\phi)u_{jI}\bom_I,
$$
and hence $\dib_{b,\phi}^*u=\dib^*_bu-\di_b(\phi)\lrcorner u$. Furthermore,
\begin{align}
[\delta^\phi_i,\bar L_j]&=\bar L_jL_i(\phi)+[L_i,\bar L_j] \nn \\
\underset{\T{by \eqref{37} and \eqref{34}}}=&\phi_{ij}-\frac12c_{ij}T(\phi)-\sum_{k=1}^nc_{ij}^kL_k(\phi)+c_{ij}T-\sum_{k=1}^n\overline{c_{ji}^k}\bar L_k+\sum_{k=1}^n c_{ij}^kL_k \nn\\
&=\phi_{ij}+c_{ij}T-\sum_{k=1}^n\overline{c_{ji}^k}\bar L_k+\sum_{k=1}^nc_{ij}^k\delta^\phi_k -\frac12c_{ij}T(\phi). \label{38}
\end{align}

The equalities \eqref{dib} and \eqref{dib*phi} lead us to Kohn-Morrey-H\"ormander inequality or basic estimate for CR manifolds. It does not function quite in the same manner as the Kohn-Morrey-H\"ormander
inequality on domains because it cannot be applied directly
to prove closed range estimates. The terms involving $T$ require significant effort to estimate. In fact, estimating the $T$ terms are the heart of the proof of Theorem \ref{thm:L^2-theory}. Equations similar in spirit to
Theorem \ref{kmh} have appeared before (e.g., \cite[Equation (12) and (10)]{HaRa11}) but the earlier versions do not apply to as wide of a class of CR manifolds as we consider here. We will not need it here,
but we could write Theorem \ref{kmh} even more generally by using the weak $Y(q)$ technology (namely, the form $\Upsilon$) from \cite{HaRa15}. We do not do that here for expositional clarity.
\begin{theorem}\label{kmh}Let $M$ be a CR manifold and $U$ a local patch. Let $\phi$ be a real $C^2$ function and   
	$q_0$ be an integer  with $0\le q_0\le n$.
	There exists a constant $C$  (independent  of $\phi$) such that for any $u\in C^\infty_{0,q}(M)$ with support in $U$,
	\begin{multline*} 
	\no{\bar{\partial}_b u}^2_{L^2_\phi}+\no{\bar{\partial}^*_{b,\phi} u}^2_{L^2_\phi}+C\no{u}^2_{L^2_\phi} 
	\geq \frac{1}{2}\Big(\sum^{q_0}_{j=1}\no{\delta_j^{\phi} u}^2_{L^2_\phi}+\sum^{n}_{j=q_0+1}\no{\bar{L}_ju}^2_{L^2_\phi} \Big) \nn \\
	+ \sum_{I\in\I_{q-1}}\sum_{i,j=1}^n (\phi_{ij}u_{iI},u_{jI})_\phi-\sum_{J\in\I_q}\sum_{j=1}^{q_0}(\phi_{jj}u_J,u_J)_\phi  \\
	+\Rre\bigg\{\sum_{I\in\I_{q-1}}\sum_{i,j=1}^n (c_{ij}Tu_{iI},u_{jI})_\phi-\sum_{J\in\I_q}\sum_{j=1}^{q_0}(c_{jj}Tu_J,u_J)_\phi \bigg\}. \nn	
	\end{multline*}
\end{theorem}

\begin{proof}
We use \eqref{dib} and compute
\begin{multline*}
\|\dbarb u\|_{L^2_\phi}^2 = \sum_{\atopp{K\in\I_{q+1}}{J,J'\in\I_q}} \sum_{k,k'=1}^n \eps^{kJ}_{k'J'} \big(\bar L_k u_J, \bar L_{k'} u_{J'}\big)_\phi +  \sum_{\atopp{K\in\I_{q+1}}{J,J'\in\I_q}}  \big(b_{JK}u_J,b_{J'K}u_{J'}\big)_\phi \\
+ 2\Rre\bigg[  \sum_{\atopp{K\in\I_{q+1}}{J,J'\in\I_q}}  \sum_{k=1}^n \eps^{kJ}_K \big(\bar L_k u_J, a_{J'K} u_{J'}\big)_\phi\bigg].
\end{multline*}
If $\eps^{kJ}_{k'J'} \neq 0$, then either $k = k'$ and $J=J'$ or $J = \{k\} \cup I$ and $J' = \{k'\} \cup I$ for some $I \in \I_{q-1}$. In the latter case
$\eps^{kJ}_{k'J'} = \eps^{kk'I}_{k'kI}=-1$. We also use the notation $(\bar L u, u)_\phi$ to denote any term of the form $(a \bar L_k u_J, u_{J'})_\phi$ or its conjugate where $a\in C^\infty(U_\mu)$. 
We use $(\delta^\phi u, u)_\phi$ to denote
any term of the same form with $\delta^\phi_j$ replacing $\bar L_k$.
It therefore follows that
\begin{align*}
\|\dbarb u\|_{L^2_\phi}^2 
&= \sum_{J\in\I_q}\sum_{k\not\in J} \|\bar L_k u_J\|_{L^2_\phi}^2 - \sum_{I\in\I_{q-1}}\sum_{\atopp{k,k'=1}{k\neq k'}}^n \big( \bar L_k u_{k' I}, \bar L_{k'} u_{kI}\big)_\phi  +(\bar L u, u)_\phi + O(\|u\|_{L^2_\phi}^2) \\
&= \sum_{J\in\I_q}\sum_{k=1}^n \|\bar L_k u_J\|_{L^2_\phi}^2 - \sum_{I\in\I_{q-1}}\sum_{k,k'=1}^n \big( \bar L_k u_{k' I}, \bar L_{k'} u_{kI}\big)_\phi  +(\bar L u, u)_\phi + O(\|u\|_{L^2_\phi}^2).
\end{align*}
A similar (but simpler) calculation shows that
\[
\|\dbars_{b,\phi} u\|_{L^2_\phi}^2 = \sum_{I\in\I_{q-1}} \sum_{j,j'=1}^n \big(\delta^\phi_j u_{jI}, \delta^\phi_{j'} u_{j'I}\big)_\phi 
+(\delta^\phi u, u)_\phi + O(\|u\|_{L^2_\phi}^2) 
\]
To proceed next, we integrate by parts and observe that
\begin{equation}\label{eqn:Lu to delta u}
\big(\bar L_k u_J, \bar L_{k'} u_{J'}\big)_\phi
= \big(\delta^\phi_{k'} u_J, \delta^\phi_k u_{J'}\big)_\phi - \big([\delta^\phi_{k'},\bar L_k] u_J, u_{J'}\big)_\phi + (\bar L u,u)_\phi + (\delta^\phi u, u)_\phi+O(\no{u}^2_{L^2_\phi}).
\end{equation}
An immediate consequence of this equality is that
\[
-\big(\bar L_k u_{jI}, \bar L_j u_{kI}\big)_\phi + \big(\delta^\phi_j u_{jI}, \delta^\phi_k u_{kI}\big)_\phi = \big([\delta^\phi_j,\bar L_k] u_{jI},u_{kI}\big)_\phi + (\bar L u, u)_\phi + (\delta^\phi u,u)_\phi+O(\no{u}^2_{L^2_\phi})
\]
and therefore
\begin{align}\label{eqn:dbarb + dbarbs proto}
\|\dbarb u\|_{L^2_\phi}^2 + \|\dbars_{b,\phi}\|_{L^2_\phi}^2 
=& \sum_{J\in\I_q}\sum_{k=1}^n \|\bar L_k u_J\|_{L^2_\phi}^2 + \Rre\bigg\{ \sum_{I,I'\in\I_{q-1}}\sum_{j,k=1}^n \big([\delta^\phi_j,\bar L_k] u_{jI},u_{kI}\big)_\phi  \bigg\}\\
&+ (\bar L u, u)_\phi + (\delta^\phi u,u)_\phi+ O(\no{u}^2_{L^2_\phi}). \nn
\end{align}
Finishing the proof requires four observations. First, using \eqref{38} on the terms $([\delta^\phi_j,\bar L_k] u_{jI},u_{kI})_\phi$ produces the off-diagonal terms involving $\phi_{jk}$ and $c_{jk}T$. Second, the on-diagonal terms
appear when (\ref{eqn:Lu to delta u}) is applied to $\|L_j u_J\|_{L^2_\phi}^2$ for $1 \leq j \leq q_0$. Third, we must control $(\bar L u, u)_\phi$ and $(\delta^\phi u,u)_\phi$, but this is a simple matter of recognizing that
$(\bar L u,u)_\phi = (\delta^\phi u, u)_\phi + O(\|u\|_{L^2_\phi}^2)$ so we can absorb all of these terms using a small constant/large constant argument where we pay the price of reducing the coefficient of the ``gradient" terms to $1/2$ and 
increasing $O(\|u\|_{L^2_\phi}^2)$. We have the result, except that it is not yet clear that the $O(\|u\|_{L^2_\phi}^2)$ term is independent of $\phi$ because the term $c_{jk} T(\phi)$ appears in (\ref{38}). However, $\{c_{jk}\}_{j,k=1}^n$ is a positive semidefinite
matrix and hence has real eigenvalues, and $T$ is a purely imaginary operator. This means
\[
\Rre\Big\{\sum_{j,k=1}^n \big(c_{jk}T(\phi) u_{jI},u_{kI}\big)_\phi\Big\} =0.
\]
The $T(\phi)$ terms that appear from the integration by parts in the second observation (the one regarding the on-diagonal terms) are handled identically.
\end{proof}

The difference between the Kohn-Morrey-H\"ormander estimate for domains  (see, e.g., \cite{Str10} or \cite{ChSh01})
and Theorem \ref{kmh} is the presence of $T$ instead of a boundary integral. It is for estimating the $T$ term that we use a microlocal argument. Specifically,
the estimate for $T$ uses a consequence of the sharp G{\aa}rding inequality. Recall the formulation from \cite{Rai10}.
\begin{proposition}\label{prop:Garding 1}
Let $R$ be a first order pseudodifferential operator such that  $\sigma(R)\geq\kappa$ where $\kappa$ is some positive
constant and $(h_{jk})$ an $n\times n$ hermitian matrix (that does not depend on $\xi$).
Then there exists a constant $C$ such that if the sum of any $q$ eigenvalues of $(h_{jk})$ is nonnegative, then
\[
\Rre \Big\{\sum_{I\in\I_{q-1}}\sum_{j,k=1}^{n}\big(h_{jk}R u_{jI}, u_{kI}\big)
 \Big \}
\geq \kappa  \Rre \sum_{I\in\I_{q-1}}\sum_{j,k=1}^{n} \big(h_{jk} u_{jI}, u_{kI}\big)
-C \|u\|_{L^2}^2,
\]
and if the the sum of any collection of $(n-q)$ eigenvalues of $(h_{jk})$ is nonnegative, then
\begin{multline*}
\Rre \Big\{\sum_{J\in\I_q}\sum_{j=1}^n \big(h_{jj}R u_J, u_J\big)
-\sum_{H\in\I_{q-1}}\sum_{j,k=1}^n\big(h_{jk}R u_{jH}, u_{kH}\big)\Big \} \\
\geq \kappa  \Rre \Big\{\sum_{J\in\I_q}\sum_{j=1}^{n} \big(h_{jj} u_J, u_J\big)
-\sum_{H\in\I_{q-1}}\sum_{j,k=1}^n \big(h_{jk} u_{jH}, u_{kH}\big) \Big \} -C \|u\|_{L^2}^2 .
\end{multline*}
\end{proposition}
Note that $(h_{jk})$ may be a matrix-valued function in $z$ but may not depend on $\xi$.

\subsection{Microlocal analysis -- the setup}
\label{s3}
To bound the terms from Theorem \ref{kmh} that involve $T$, we continue to work on smooth forms that are supported in a small  neighborhood $U\subset M$. Our approach is microlocal and we
adopt the familiar setup introduced by Kohn \cite{Koh86,Koh02}. See also Nicoara \cite{Nic06} and Raich \cite{Rai10}.

Denote the coordinates $\R^{2n+1}$ by 
$x=(x',x_{2n+1})=(x_1,...x_{2n},x_{2n+1})$ with the origin at some $x_0\in U$. We can arrange the coordinates so that if $z_j=x_j+\sqrt{-1}x_{j+n}$ for $j=1,\dots,n$, 
then  $L_j|_{z_0}=\frac{\di}{\di z_j}|_{z_0}$ for $j=1,\dots,n$, and $T=-\sqrt{-1}\frac{\di}{\di x_{2n+1}}$.  Let $\xi=$ $(\xi_1,...,\xi_{2n+1})$ $=(\xi',\xi_{2n+1})$ be the dual coordinates to $x$ in Fourier space.

 Let $\mathcal C^+,\mathcal C^-,\mathcal C^0$ be a covering of $\R^{2n+1}$ so that
 \begin{eqnarray}\begin{split}
\mathcal C^+=&\big\{\xi : \xi_{2n+1}> \frac{1}{4} |\xi'| \big\}\cap \{\xi : |\xi|\ge 1\};\\
\mathcal C^-=&\{\xi : \xi\in \mathcal C^+\};\\
\mathcal C^0=&\{\xi : | \xi_{2n+1}|<\frac{3}{4}|\xi'| \}\cup \{\xi :|\xi|<3\}.
\end{split}
\end{eqnarray}
For the remainder of this paper, let $\psi$ be a smooth function so that $\psi \equiv 1$ on $\{\xi : \xi_{2n+1} > \frac 13 |\xi'|\} \cap \{\xi : |\xi| \geq 2\}$ and $\supp \psi \subset \opC^+$.  It follows from the definitions
of $\opC^+$, $\opC^0$, and $\psi$, that $\supp d\psi \subset \opC^0$.
Define
\[
\psi^+(\xi):=\psi(\xi), \qquad \psi^-(\xi):=\psi(-\xi), \qquad \psi^0(\xi):=\sqrt{1-(\psi^+(\xi))^2-(\psi^-(\xi))^2}. 
\]
Let $\tilde\psi^0$ be a smooth function that dominates $\psi^0$ in the sense that 
$\supp\tilde\psi^0\subset \mathcal C^0$ and $\tilde\psi^0=1$ on a neighborhood of $\supp \psi^0\cup\supp(d\psi^+) \cup \supp (d\psi^-)$.

Associated to the smooth function $\psi$ is a pseudodifferential operator $\Psi$ whose symbol $\sigma(\Psi)=\psi$. This means that if $\varphi\in C^\infty_c(U)$, then
$$\widehat{\Psi \varphi}(\xi)=\psi(\xi)\hat{\varphi}(\xi),$$
where $\hat{}$ denotes the Fourier transform. The operators $\Psi^+$, $\Psi^-$, $\Psi^0$, and $\tilde\Psi^0$ are defined analogously, with symbols $\psi^+$, $\psi^-$, $\psi^0$, and $\tilde\psi^0$, respectively. 

By construction, $(\psi^+)^2 + (\psi^-)^2 + (\psi^0)^2 =1$, from which it follows immediately that 
$(\Psi^+)^*\Psi^+ + (\Psi^-)^*\Psi^- + (\Psi^0)^*\Psi^0 = Id$, the identity operator.

For the proof of Theorem~\ref{thm:L^2-theory}, we will need dilated versions  $\Psi^\bullet$ and $\tilde\Psi^0$ where the superscript $\bullet$ means $+$, $-$, or $0$.
Let $A\ge1$ (chosen later). Let $\Psi^\bullet_A$ and $\tilde\Psi^0_A$ be the pseudodifferential operators with symbol $\psi^\bullet_A(\xi) = \psi^\bullet(\xi/A)$ and $\tilde\psi^\bullet_A(\xi) = \tilde\psi^\bullet(\xi/A)$, respectively. We say that a cutoff function
$\zeta$ dominates a cutoff function $\zeta'$ and denote by $\zeta'\prec\zeta$ if $\zeta \equiv 1$ on $\supp\zeta'$.
We write the next several results for a generic $A$ but will use $A:=t$ in the proof of Theorem \ref{thm:global regularity} and $A:=A_\eps$ in the proof of Theorem \ref{thm:CRpsh (0,1)}. The next result follows immediately from
Proposition \ref{prop:Garding 1} and the arguments of Lemma 4.6 and Lemma 4.7 in \cite{Rai10}.
\begin{lemma}
	\Label{l3.2} Let $M$ be a pseudoconvex CR manifold and $U$ be a local patch of $M$. For $\phi\in C^\infty(M)$, $u\in C^\infty_{0,q}(M)$, and cutoff functions $\zeta\prec\tilde\zeta\prec \zeta'$ on $U$, we have
	\begin{eqnarray}\label{Psi+}
	\begin{split}
(i)	\qquad\T{\normalfont Re}\Big((d\gamma\lrcorner T\tilde\zeta \Psi^+_A\zeta u,\tilde\zeta \Psi^+_A\zeta u)_\phi\Big)
	\ge &A(d\gamma\lrcorner \tilde\zeta \Psi^+_A\zeta u,\tilde\zeta \Psi^+_A\zeta u)_\phi\\ -&c\no{\tilde\zeta \Psi^+_A\zeta u}_{L^2_\phi}^2-c_{A,\phi}\no{ \zeta'\tilde\Psi^0_A\zeta u}_{L^2}^2
	\end{split}
	\end{eqnarray}
	for any $q=1,\dots, n$; and
		\begin{eqnarray}\label{Psi-}
		\begin{split}
(ii)		\qquad\T{\normalfont Re}&\Big((d\gamma\lrcorner T\tilde\zeta \Psi^-_A\zeta u,\tilde\zeta \Psi^-_A\zeta u)_{-\phi}-(\Tr(d\gamma) T\tilde\zeta \Psi^-_A\zeta u,\tilde\zeta \Psi^-_A\zeta u)_{-\phi}\Big)\\
		\ge &A\Big(\Tr(d\gamma) \tilde\zeta \Psi^-_A\zeta u,\tilde\zeta \Psi^-_A\zeta u)_{-\phi}-(d\gamma\lrcorner \tilde\zeta \Psi^-_A\zeta u,\tilde\zeta \Psi^-_A\zeta u)_{-\phi}\Big)\\
		& -c\no{\tilde\zeta \Psi^-_A\zeta u}^2_{L^2_{-\phi}}-c_{A,\phi}\no{ \zeta'\tilde\Psi^0_A\zeta u}_{L^2}^2
		\end{split}
		\end{eqnarray}
		for any $q=0,1,\dots, n-1$. Here,  $c$ (resp. $c_{A,\phi}$)  is a positive constant independent (resp. dependent) of  $\phi$.
\end{lemma}

In combination with the Kohn-Morrey-H\"omander inequality, Lemma~\ref{l3.2} yields 
\begin{corollary}
	\label{c3.3} Let $M$ be a pseudoconvex CR manifold and $U$ be a local patch of $M$. For $\phi\in C^\infty(M)$, $u\in C^\infty_{0,q}(M)$, and cutoff function $\zeta\prec\tilde\zeta\prec \zeta'$ on $U$, then we have
	\begin{eqnarray}
	\Label{nova2}
	\begin{split}
	c_{A,\phi}\no{\zeta'\tilde\Psi_A^0\zeta u}_{L^2}^2+&c\left(\no{\tilde \zeta\Psi_A^+\zeta u}_{L^2_\phi}^2+\no{\bar{\partial}_b \tilde \zeta\Psi_A^+\zeta u}_{L^2_\phi}^2+\no{\bar{\partial}^*_{b,\phi}\tilde \zeta\Psi_A^+\zeta u}_{L^2_\phi}^2\right)\\
	\ge &\left((\L_\phi+Ad\gamma)\lrcorner \tilde \zeta\Psi_A^+\zeta u,\tilde \zeta\Psi_A^+\zeta u\right)_\phi
	\end{split}
	\end{eqnarray}
	for any $n=1,\dots, n$; and
	\begin{eqnarray}
	\Label{nova3}
	\begin{split}
	c_{A,\phi}\no{\zeta'\tilde\Psi_A^0\zeta u}_{L^2}^2+&c\left(\no{\tilde \zeta\Psi_A^-\zeta u}^2_{L^2_{-\phi}}+\no{\bar{\partial}_b \tilde \zeta\Psi_A^-\zeta u}^2_{L^2_{-\phi}}+\no{\bar{\partial}^*_{b,-\phi}\tilde \zeta\Psi_A^-\zeta u}^2_{L^2_{-\phi}}\right)\\
	\ge &\Big([\Tr(\L_\phi)+A\Tr(d\gamma)]\times \tilde\zeta \Psi^-_A\zeta u,\tilde\zeta \Psi^-_A\zeta u\Big)_{-\phi}\\
	&- \left([\L_\phi+Ad\gamma]\lrcorner \tilde \zeta\Psi_A^{-}\zeta u,\tilde \zeta\Psi_A^-\zeta u\right)_{-\phi}
	\end{split}
	\end{eqnarray}
	for any $q=0,\dots,n-1$.
\end{corollary}


%
%

\section{The $L^2$-Sobolev Theory for $\Boxb$ and the proof of Theorem \ref{thm:L^2-theory}}\label{sec:proving L^2 theory}
\subsection{The existence to the solution of the $\Box_b$ when $n\geq 2$ and $1\leq q \leq n-1$.}
We now assume that $M$ is endowed a with a smooth function $\lambda$ that is strictly CR-plurisubharmonic on $(0,q_0)$-forms whose defining inequality is given by (\ref{CRpluri}).
The function $\lambda$ is $q_0$-compatible in the language of Harrington and Raich \cite{HaRa11}, and we can follow their argument nearly
verbatim to establish a weighted $L^2$-theory (compare with the proof of \cite[Theorem 1.2]{HaRa11}). 

Observe that if the inequality (\ref{CRpluri}) holds for $q_0$, then it holds  for any $q\ge q_0$. Additionally,
\begin{eqnarray}\label{new2}\begin{split}
\la[\Tr(\L_{\lambda})+\Tr(d\gamma)] v,\bar v\ra -\la(\L_{\lambda}+d\gamma)\lrcorner v, \bar v\ra \ge a|v|^2\end{split}
\end{eqnarray}
for all $(0,q)$-forms $v$ with $q\le n-q_0$.  

For each $t\ge 1$, we 
use Corollary~\ref{c3.3} with $\phi=t\lambda$ to obtain
	\begin{eqnarray}
	\Label{local1}
	\begin{split}
	c_{t}\no{\zeta'\tilde\Psi_t^0\zeta u}_{L^2}^2+&c\left(\no{\tilde \zeta\Psi_t^+\zeta u}^2_{L^2_{t\lambda}}+\no{\bar{\partial}_b \tilde \zeta\Psi_t^+\zeta u}^2_{L^2_{t\lambda}}+\no{\bar{\partial}^*_{b,t\lambda}\tilde \zeta\Psi_t^+\zeta u}^2_{L^2_{t\lambda}}\right)\\
	\ge &\left((\L_{t\lambda}+td\gamma)\lrcorner \tilde \zeta\Psi_t^+\zeta u,\tilde \zeta\Psi_t^+\zeta u\right)_{t\lambda}\\
	\ge & at\no{\tilde \zeta\Psi_t^+\zeta u}^2_{L^2_{t\lambda}}
		\end{split}
		\end{eqnarray}
for any $u\in C^\infty_{0,q}(M)$ with $q\ge q_0$. Analogously, we also have
	\begin{eqnarray}
	\Label{local2}
	\begin{split}	c_{t}\no{\zeta'\tilde\Psi_t^0\zeta u}_{L^2}^2+&	c\left(\no{\tilde \zeta\Psi_t^-\zeta u}^2_{L^2_{-t\lambda}}+
	\no{\bar{\partial}_b \tilde \zeta\Psi_t^-\zeta u}^2_{L^2_{-t\lambda}}+\no{\bar{\partial}^*_{b,-t\lambda}\tilde \zeta\Psi_t^-\zeta u}^2_{L^2_{-t\lambda}}\right)\\
	\ge &\, a t\no{\tilde \zeta\Psi_t^-\zeta u}^2_{L^2_{-t\lambda}}
	\end{split}
	\end{eqnarray}
	for any $u\in C^\infty_{0,q}(M)$ with $ q\le n-q_0$. This estimate holds for cutoff functions $\zeta,\tilde\zeta,\zeta'$ having compact support on a local patch $U$ of $M$. In order to 
	prove a  global estimate, we let $\{U_\nu\}$ be a cover of $M$ and $\{\zeta_\nu\}$ 
	be a partition of unity subordinate to $\{U_\nu\}$. Supported on each $U_\nu$ are the pseudodifferential operators $\Psi^{\cdot,\nu}_t$ and $\tilde\Psi^{0,\nu}_t$ 
	where $\cdot$ represents $+$, $-$, or $0$. For each $\zeta_\nu$, let $\tilde\zeta_\nu$ be a cutoff function that dominates $\zeta_\nu$. 
	We define an inner product and norm that are well-suited to estimates using microlocal analysis. Set
\begin{multline*}
\la| u,v |\ra_{t\lam} 
= \sum_\nu \Big[ ( \tilde\zeta_\nu\Psi^+_t\zeta_\nu u^\nu, \tilde\zeta_\nu\Psi^+_t\zeta_\nu v^\nu )_{L^2_{t\lambda}} 
+ (\tilde\zeta_\nu \Psi^0_t \zeta_\nu u^\nu ,\tilde\zeta_\nu \Psi^0_t \zeta_\nu v^\nu )
+ (\tilde\zeta_\nu \Psi^-_t \zeta_\nu u^\nu, \tilde\zeta_\nu \Psi^-_t \zeta_\nu v^\nu )_{-t\lambda} \Big]
\end{multline*}
and 
\[
\la| u |\ra_{t\lambda} ^2 
= \sum_{\nu} \Big[ \| \tilde\zeta_\nu \Psi^+_t \zeta_\nu u^\nu \|_{L^2_{t\lambda}}^2 + \|\tilde\zeta_\nu \Psi^0_t \zeta_\nu u^\nu \|_{L^2}^2
+ \|\tilde\zeta_\nu \Psi^-_t \zeta_\nu u^\nu \|_{L^2_{-t\lambda}}^2 \Big],
\]
where $u^\nu$ is the form $u$ expressed in the local coordinates on $U_\nu$. The superscript $\nu$ will often be omitted. We denote the adjoint of $\dbarb$ with respect to this norm
by $\dbarb^{*,t}$ and the associated quadratic form
\[
Q_{b,t\lambda}\la| u,v|\ra = \la| \dbar u, \dbar v |\ra_{t\lambda} + \la| \dbarb^{*,t} u, \dbarb^{*,t} v |\ra_{t\lambda}.
\]

The space of $t\lambda$-harmonic forms $\opH^q_{t\lambda}(M)$ is
\[
\opH^q_{t\lambda}(M) = \{ u \in L^2_{0,q}(M) : Q_{b,t\lambda}\la|u,u|\ra =0\}.
\]
By using the pseudoconvex-oriented hypothesis, the estimates \eqref{local1}-\eqref{local2} for $U_\nu$, and
the well-known elliptic estimate for $\Psi^0$, it follows that there exists $T_0>0$ such that for any $t\ge t_0$ the estimate 
\begin{eqnarray}\label{L2}
\la| u |\ra_{t\lambda} ^2 \le \frac{c}{t}Q_{b,t\lambda}\la|u,u|\ra+c_t\no{u}_{H^{-1}}^2
\end{eqnarray} 
holds for all $u\in C^\infty_{0,q}(M)$ with $q_0\le q\le n-q_0$. See \cite{Nic06,HaRa11} for details.

For a form $u$ on $M$, the Sobolev norm of order $s$ is given by the following:
\[
\|u\|_{H^s}^2 = \sum_\nu \|\tilde\zeta_\nu\Lambda^s\zeta_\nu \vp^\nu \|_{L^2}^2
\]
where $\Lambda$ is defined to be the pseudodifferential operator with symbol $(1+|\xi|^2)^{1/2}$.

As in \cite{HaRa11}, we can also bring the estimate \eqref{L2} to higher order Sobolev indices: for each $s\ge0$, there exists $T_s>0$ such that for any $t\ge T_s$, 
 \begin{eqnarray}\label{Hs}
 \la|\Lambda^s u |\ra_{t\lambda} ^2 \le \frac{c}{t}\left(\la|\Lambda^s\dib_bu|\ra_{t\lambda}^2+\la|\Lambda^s\dib_b^{*,t}u|\ra_{t\lambda}^2\right)+c_t\no{u}_{H^{s-1}}^2
 \end{eqnarray} 
 holds for all $u\in C^\infty_{0,q}(M)$ with $q_0\le q\le n-q_0$. 
In \cite{Nic06}, it is shown that  there exist constants $c_t$ and $C_t$ so that
\begin{equation}\label{eqn:norm equivalence}
c_t \| u\|_{L^2}^2 \leq \la| u|\ra_{t\lambda}^2  \leq C_t \|u\|_{L^2}^2
\end{equation}
where $c_t$ and $C_t$ depend on $\max_{M}|\lambda|$. We thus have closed range estimates for  $\dbarb:H^s_{0,q}(M)\to H^s_{0,q+1}(M)$ and $\dbarb^{*,t}:H^s_{0,q}(M)\to H^s_{0,q-1}(M)$.
The following theorem now follows from the arguments of \cite{HaRa11}.

%
%
\begin{theorem}\label{thm:main theorem for weighted spaces}
Let $M^{2n+1}$ be an abstract CR manifold that is pseudoconvex-oriented and admits a smooth function $\lambda$ that is strictly CR plurisubharmonic on $(0,q_0)$-forms for some $1 \leq q_0 \leq \frac n2$. 
Then for all $q_0 \leq q \leq n-q_0$
and $s\geq 0$, there exists $T_s\geq 0$
so that the following hold:
\begin{enumerate}\renewcommand{\labelenumi}{(\roman{enumi})}
\item The operators $\dbarb: L^2_{0,q}(M)\to L^2_{0,q+1}(M)$ and $\dbarb: L^2_{0,q-1}(M)\to L^2_{0,q}(M)$ have closed range
with respect to $\la|\cdot|\ra_{t\lam}$. Additionally, for any $s>0$ if $t\geq T_s$, then
$\dbarb: H^s_{0,q}(M)\to H^s_{0,q+1}(M)$ and $\dbarb: H^s_{0,q-1}(M)\to H^s_{0,q}(M)$ have closed range with respect to $\la|\Lambda^s\cdot|\ra_{t\lam}$;
\item  The operators $\dbarb^{*,t}: L^2_{0,q+1}(M)\to L^2_{0,q}(M)$ and $\dbarb^{*,t}: L^2_{0,q}(M)\to L^2_{0,q-1}(M)$ have closed range
with respect to $\la|\cdot|\ra_{t\lam}$. Additionally, if $t\geq T_s$, then
$\dbarb^{*,t}: H^s_{0,q+1}(M)\to H^s_{0,q}(M)$ and $\dbarb^{*,t}: H^s_{0,q}(M)\to H^s_{0,q-1}(M)$ have closed range with respect to $\la|\Lambda^s\cdot|\ra_{t\lam}$;
\item The Kohn Laplacian defined by $\Boxb^t = \dbarb\dbarb^{*,t} + \dbarb^{*,t}\dbarb$ has closed range on $L^2_{0,q}(M)$
(with respect to $\la|\cdot|\ra_{t\lam}$) and also on
$H^s_{0,q}(M)$ (with respect to $\la|\Lambda^s\cdot|\ra_{t\lam}$) if $t\geq T_s$;
\item The space of harmonic forms $\H^q_t(M)$, defined to be the $(0,q)$-forms annihilated by $\dbarb$ and $\dbarb^{*,t}$ is finite dimensional;
\item The complex Green operator $G_{q,t}$ is continuous on $L^2_{0,q}(M)$
(with respect to $\la|\cdot|\ra_{t\lam}$) and also on
$H^s_{0,q}(M)$ (with respect to $\la|\Lambda^s\cdot|\ra_{t\lam}$) if $t\geq T_s$;
\item The canonical solution operators
for $\dbarb$, $\dbarb^{*,t} G_{q,t}:L^2_{0,q}(M)\to L^2_{0,q-1}(M)$ and $G_{q,t}\dbarb^{*,t} : L^2_{0,q+1}(M)\to L^2_{0,q}(M)$ are continuous (with respect to
$\la|\cdot|\ra_{t\lam}$).
Additionally,\\ $\dbarb^{*,t} G_{q,t}:H^s_{0,q}(M)\to H^s_{0,q-1}(M)$ and $G_{q,t}\dbarb^{*,t} : H^s_{0,q+1}(M)\to H^s_{0,q}(M)$
are continuous (with respect to $\la|\Lambda^s\cdot|\ra_{t\lam}$) if $t\geq T_s$.
\item The canonical solution operators for $\dbarb^{*,t}$,
$\dbarb G_{q,t}:L^2_{0,q}(M)\to L^2_{0,q+1}(M)$
and $G_{q,t}\dbarb : L^2_{0,q-1}(M)\to L^2_{0,q}(M)$ are continuous (with respect to $\la|\cdot|\ra_{t\lam}$).
Additionally,\\ $\dbarb G_{q,t}:H^s_{0,q}(M)\to H^s_{0,q+1}(M)$
and $G_{q,t}\dbarb : H^s_{0,q-1}(M)\to H^s_{0,q}(M)$ are continuous (with respect to $\la|\Lambda^s\cdot|\ra_{t\lam}$) if $t\geq T_s$.
\item The Szeg\"o projections $S_{q,t} = I - \dbarb^{*,t}\dbarb G_{q,t}$ and $S_{q-1,t} = I - \dbarb^{*,t} G_{q,t} \dbarb$ are continuous on
$L^2_{0,q}(M)$ and $L^2_{0,q-1}(M)$, respectively and with respect to $\la|\cdot|\ra_{t\lam}$. Additionally,
if $t\geq T_s$, then $S_{q,t}$ and $S_{q-1,t}$ are continuous on
$H^s_{0,q}$ and $H^s_{0,q-1}$ (with respect to $\la|\Lambda^s\cdot|\ra_{t\lam}$), respectively.
\item  If $\tilde q = q$ or $q+1$ and
$\alpha\in H^s_{0,q}(M)$ so that  $\dbarb \alpha =0$ and $\alpha \perp \H^{\tilde q}_t$ (with respect to $\la|\cdot|\ra_{t\lam}$),
then there exists $u\in H^s_{0,\tilde q-1}(M)$ so that
\[
\dbarb u = \alpha;
\]
\item  If $\tilde q = q$ or $q+1$ and $\alpha\in C^\infty_{0,\tilde q}(M)$ satisfies $\dbarb\alpha=0$ and $\alpha \perp \H^{\tilde q}_t$
(with respect to $\la| \cdot, \cdot |\ra_{t\lam}$),
then there exists $u\in C^\infty_{0,\tilde q-1}(M)$ so that
\[
\dbarb u = \alpha.
\]
\end{enumerate}
\end{theorem}

Turning to the proof of Theorem \ref{thm:L^2-theory}, a consequence of Theorem \ref{thm:main theorem for weighted spaces} and \eqref{eqn:norm equivalence} is that 
$\dbarb:L^2_{0,\tilde q}(M)\to L^2_{0,\tilde q+1}(M)$ has closed range, $\tilde q = q$ or $q-1$.
Functional analysis shows that the $L^2$-adjoint operators
$\dbarbs:L^2_{0,\tilde q+1}(M) \to L^2_{0,\tilde q}(M)$ also has closed range \cite[Theorem 1.1.1]{Hor65}. Additionally, the finite dimensionality of $\mathcal H^q_t(M)$ combined with the closed range
of $\dbarb$ on $L^2_{0,\tilde q}(M)$, $\tilde q = q,q-1$ 
implies the finite dimensionality of the
unweighted space of harmonic forms $\mathcal H_{0,q}(M)$. While this fact is likely well-known, Straube and Raich give a proof in \cite[p.772]{RaSt08}.

The cases $q=0$ and $q=n$ (when $n\geq 2$) follow easily from the formulas $G_0 = \dbarbs G_1^2 \dbarb$ and $G_n = \dbar G_{n-1}^2 \dbarbs$ and the already proven parts of the theorem.
This concludes the proof of Theorem \ref{thm:L^2-theory}.

%
%
\section{Global hypoellipticity of $\Box_b$ and the Proof of Theorem \ref{thm:CRpsh (0,1)}}\label{sec:hypoellipticity}
\subsection{A weak compactness estimate for $\Box_b$}\label{s6}
In this section, we assume: i) for any $\eps>0$ there exist a vector $T_\eps$ transversal to $T^{1,0}M\oplus T^{0,1}M$ such that $0<c_1<\gamma(T_\eps)<c_2$ uniformly in $\eps$ and 
ii) there exists a  covering $\{U_\eta\}$ by local patches such that on each $U:=U_\eta$ there exists  $\lambda_\eps:=\lambda^\eta_\eps$  so that $\lambda_\eps$ is uniformly bounded and  
$$\la(\L_{\lambda_\eps}+A_\eps d\gamma)\lrcorner u,\bar u\ra\ge \frac{1}{\eps}|\alpha_\eps|^2|u|^2$$
holds on $U$ for all $(0,q_0)$-forms $u\in C^\infty_{0,q}(M)$. Here, $\alpha_\eps=-\{\T{Lie}\}_{T_\eps}(\gamma)$. 
In Section \ref{sec:proving L^2 theory}, we proved estimates for weighted operators in Sobolev spaces. Now, under
our stronger assumption of the existence of $T_\eps$, we will prove estimates in Sobolev spaces for the unweighted system
$(\dib_b,\dib_b^*)$. In order to do that, we use  the composition weight $\phi=\chi(\lambda_\eps)$ for a smooth function $\chi:\R\to \R$ chosen later but satisfying $\dot\chi,\ddot{\chi}> 0$.
By the definition of the Levi form , it follows that
$$\L_{\chi(\lambda_\eps)}=\dot\chi \L_{\lambda_\eps}+\ddot\chi \di_b\lambda_\eps\we\dib_b\lambda_\eps,$$
and hence 
\begin{eqnarray} \label{notice1}\begin{split}
\la (\L_{\chi(\lambda_\eps)}+\dot\chi A_\eps d\gamma)\lrcorner u,\bar u\ra =&\dot\chi\la (\L_{\lambda_\eps}+A_\eps d\gamma)\lrcorner u,\bar u\ra +\ddot\chi|\di_b\lambda_\eps\lrcorner u|^2\\
\ge&\frac{1}{\eps}|\sqrt{\dot{\chi}}\alpha_\eps|^2|u|^2+\ddot\chi|\di_b\lambda_\eps\lrcorner u|^2. 
\end{split}\end{eqnarray}
We also notice that 
\begin{eqnarray}
\label{notice2}|\dib^*_{b,\chi(\phi^\eps)}u|^2 \le 2 |\dib_b^*u|^2+2\dot\chi^2|\di_b\lambda_\eps\lrcorner u|^2.
\end{eqnarray}
Now we use Corollary,~\ref{c3.3}(i) for $\phi:=\chi(\lambda_\eps)$, $A:=\dot\chi(\lambda_\eps) A_\eps$ and plug in \eqref{notice1} and \eqref{notice2} into \eqref{nova2} to obtain
	\begin{eqnarray}\label{400}
	\begin{split}
	c_{\eps}\no{\zeta'\tilde\Psi_{A}^0\zeta u}_{L^2}^2+&c\left(\no{\tilde \zeta\Psi_{A}^+\zeta u}_{L^2_\phi}^2+2\no{\dot\chi\di_b\lambda_\eps\lrcorner \tilde\zeta\Psi_{A}^+\zeta u}_{L^2_\phi}^2+\no{\bar{\partial}_b \tilde \zeta\Psi_{A}^+\zeta u}_{L^2_\phi}^2+\no{\bar{\partial}^*_{b,\phi}\tilde \zeta\Psi_{A}^+\zeta u}_{L^2_\phi}^2\right)\\
	&\ge \frac{1}{\eps}\no{\sqrt{\dot\chi}|\alpha_\eps|\tilde \zeta\Psi_{A}^+\zeta u}_{L^2_\phi}^2+\no{\sqrt{\ddot{\chi}}\di_b\lambda_\eps\lrcorner \tilde \zeta\Psi_{A}^+\zeta u}_{L^2_\phi}^2
	\end{split}
	\end{eqnarray}
	for any $u\in C^\infty_{0,q}(M)$ with $q=q_0,\dots,n$. The function $\lambda_\eps$ is uniformly bounded, so we may assume that $|\lambda_\eps|\le 1$. 
	Thus, if we choose $\chi(t)=\frac{1}{2c}e^{t-1}$ then $\ddot\chi(t)\ge 2c \dot\chi^2(t)$ for $|t|\le 1$, and hence we can absorb 
	$2c\no{\dot\chi\di_b\lambda_\eps\lrcorner \tilde\zeta\Psi_{A_\eps}^+\zeta u}^2_{\chi(\lambda_\eps)}$ by the RHS. By this choice of $\chi$, we also get a uniform bound for 
	$e^{-\chi}$ and $\dot{\chi}\ge \frac{1}{2e^2c}$. Consequently, we can remove the weight from both sides of \eqref{400} and obtain
	\begin{eqnarray}\label{400b}
	\begin{split}
	c_{\eps}\no{\zeta'\tilde\Psi_{A}^0\zeta u}_{L^2}^2+&c\left(\no{\tilde \zeta\Psi_{A}^+\zeta u}_{L^2}^2+\no{\bar{\partial}_b \tilde \zeta\Psi_{A}^+\zeta u}_{L^2}^2+\no{\bar{\partial}^*_{b}\tilde \zeta\Psi_{A}^+\zeta u}_{L^2}^2\right)\\
	&\ge \frac{1}{\eps}\no{|\alpha_\eps|\tilde \zeta\Psi_{A}^+\zeta u}_{L^2}^2
	\end{split}
	\end{eqnarray}
	for any $u\in C^\infty_{0,q}(M)$ with $q\ge q_0$.\\

To bound the $\Psi^-$ terms, we cannot 
use an analogous argument. The problem is that there is no  $|\di_b\lambda_\eps\lrcorner u|^2$ term to absorb unwanted terms.
Indeed,
\begin{eqnarray} \label{no211}\begin{split}
&\la [\Tr(\L_{\chi(\lambda_\eps)})+\dot\chi A\Tr(d\gamma)]\times u,\bar u\ra-\la (\L_{\chi(\lambda_\eps)}+\dot\chi A_\eps d\gamma)\lrcorner u,\bar u\ra \\
&=\dot\chi\left(\la [\Tr(\L_{\lambda_\eps})+A_\eps\Tr(d\gamma)]\times u,\bar u\ra-\la (\L_{\lambda_\eps}+A_\eps d\gamma)\lrcorner u,\bar u\ra\right)\\
	&= \ddot\chi\left(\la\Tr(\di_b\lambda_\eps\we \dib_b\lambda_\eps)\times u,\bar u\ra -|\di_b\lambda_\eps\lrcorner u|^2\right)\\
&\ge \frac{1}{\eps}|\sqrt{\dot{\chi}}\alpha_\eps|^2|u|^2+\ddot\chi|\dib_b\lambda_\eps\we u|^2; 
\end{split}\end{eqnarray}
and  the $|\dib_b\lambda_\eps\we u|$ cannot absorb $|\di_b\lambda_\eps\lrcorner u|$ in general. Instead, we can obtain the estimate for $\Psi^-$ by a Hodge-$*$ argument (see \cite{Koh02, Kha16b}). 
Indeed, using the ideas in \cite[Theorem 5]{Kha16b}, it follows that \eqref{400b} is equivalent to 
	\begin{eqnarray}\label{401}
	\begin{split}
	c_{\eps}\no{\zeta'\tilde\Psi_{A}^0\zeta u}_{L^2}^2+&c\left(\no{\tilde \zeta\Psi_{A}^-\zeta u}_{L^2}^2+\no{\bar{\partial}_b \tilde \zeta\Psi_{A}^-\zeta u}_{L^2}^2+\no{\bar{\partial}^*_{b}\tilde \zeta\Psi_{A}^-\zeta u}_{L^2}^2\right)\\
	&\ge \frac{1}{\eps}\no{|\alpha_\eps|\tilde \zeta\Psi_{A}^-\zeta u}_{L^2}^2
	\end{split}
	\end{eqnarray}
	for any $u\in C^\infty_{0,q}(M)$ with $q\le n-q_0$.\\
	To obtain our global estimates, we use \eqref{400b} and \eqref{401} on each local patch $U_\eta$ of the covering $\{U_\eta\}$, together with the elliptic estimate of the $\Psi^0$-terms and 
	$$\sum_{\eta}\left(\no{[\bar{\partial}_b, \tilde \zeta\Psi_{A}^\pm\zeta] u}_{L^2}^2+\no{[\bar{\partial}_b, \tilde \zeta\Psi_{A}^\pm\zeta ]u}_{L^2}^2\right)
	\le c\left(\no{u}_{L^2}^2+\no{\dib_bu}_{L^2}^2+\no{\dib_b^*u}_{L^2}^2\right)+c_\eps\no{u}^2_{H^{-1}}.$$ We then see
	\begin{eqnarray}
	\begin{split}
	& \frac{1}{\eps}\no{|\alpha_\eps|u}_{L^2}^2\le c\left(\no{\dib_bu}_{L^2}^2+\no{\dib_b^*u}_{L^2}^2+\no{u}_{L^2}^2\right)+c_\eps\no{u}^2_{H^{-1}}
	\end{split}
	\end{eqnarray}
	for any $u\in C^\infty_{0,q}(M)$ with $q_0\le q\le n-q_0$. If $M$ admits a strictly CR-plurisubharmonic function on $(0,q_0)$-forms, we have already proved that 
	$$\no{u}_{L^2}^2\le c\left(\no{\dib_bu}_{L^2}^2+\no{\dib_b^{ *}u}_{L^2}^2\right)+c'\no{u}^2_{H^{-1}}$$
for any $u\in C^\infty_{0,q}(M)$ with $q_0\le q\le n-q_0$. Thus, we have the following theorem.
\begin{theorem}
	Assume that the hypothesis of Theorem~\ref{thm:global regularity} holds. Then for any $\epsilon>0$ there exist a vector $T_\epsilon$ and  a constants $c_\epsilon>0$ such that $c_1\le \gamma(T_\epsilon)\le c_2$ uniformly in $\epsilon$ and 
	\begin{eqnarray}\label{Theta+}
\frac{1}{\eps}\big\||\alpha_\eps|u\big\|_{L^2}^2+\no{u}_{L^2}^2\le c \left(\no{\dib_b u}_{L^2}^2+\no{\dib_b^*u}_{L^2}^2\right)+c_\epsilon \no{u}^2_{H^{-1}}
	\end{eqnarray}
	holds for all $u\in C^\infty_{0,q}(M)$ with $q_0\le q\le n-q_0$ and $\alpha_\eps=-\{\T{Lie}\}_{T_\eps}(\gamma)$.
\end{theorem}

\subsection{Global hypoellipticity for $\Box_b$}
\Label{s4}
Let $s\geq 0$ be an integer and $D^s$ denote a differential operator of order $s$, $\nabla_b=(\nabla_b^{1,0},\nabla_b^{0,1})$ where $\nabla_b^{1,0}=(L_1,\dots, L_n)$ and $\nabla_b^{0,1}=(\bar L_1,\dots,\bar L_n)$.
By the Kohn-Morrey-H\"ormander inequality, for $s\ge 1$, $\epsilon>0$ there exists $c_{\epsilon,s}>0$ such that   
\begin{eqnarray}\label{421}
\begin{split}\no{\nabla_bu}_{H^{s-1}}^2\le&
c_s\left(\no{\nabla_b D^{s-1}u}_{L^2}^2+\no{u}_{H^{s-1}}^2\right)\\
\le& c_{s}\left( \no{\dib_b D^{s-1}u}_{L^2}^2+\no{\dib_b^* D^{s-1}u}_{L^2}^2+\no{u}^2_{H^{s-1}}+\no{u}_{H^s}\no{u}_{H^{s-1}}\right)\\
\le& c_{s}\left(\no{\dib_bu}^2_{H^{s-1}}+\no{\dib_b^*u}^2_{H^{s-1}}\right)+c_{s,\epsilon}\no{u}_{H^{s-1}}^2+\eps\no{u}^2_{H^s},
\end{split}
\end{eqnarray}
Let $T_\eps$ be a global, purely imaginary vector field transversal to $T^{1,0}M\oplus T^{0,1}M$ such that $0<c_1\le\gamma(T_\eps)\le c_2$. 
There exist a function $b_\epsilon$ and a vector field $X_\epsilon\in T^{1,0}M\oplus T^{0,1}M$ such that $0<c_1\le b_\eps\le c_2$ and
$$T=b_\eps T_\eps+X_\eps$$
This implies that for any $s\ge 1$, there exist
constants $c>0$ and $c_{\epsilon,s}>0$ such that 
\begin{eqnarray}
\label{423}
\no{T^su}_{L^2}^2\le c\no{T^s_\eps u}_{L^2}^2+c_{\eps, s}\left(\no{\nabla_b u}^2_{H^{s-1}}+\no{u}_{H^{s-1}}^2\right) + \eps\|u\|_s^2
\end{eqnarray}
From \eqref{421}, \eqref{423} and $\no{u}^2_{H^s}\le c_s\no{\nabla_bu}_{L^2}^2+c\no{T^su}_{L^2}^2$, we get the reduction from $D^s$ to $T^s_\epsilon$ by the inequality
\begin{eqnarray}
\label{424}
\no{u}_{H^s}^2\le c_{s,\eps}\left(\no{\dib_bu}^2_{H^{s-1}}+\no{\dib_b^*u}^2_{H^{s-1}}+\no{u}_{H^{s-1}}^2\right)+c\no{T_\epsilon^su}_{L^2}^2. 
\end{eqnarray}
Moreover, since
$$\no{\dib_bu}^2_{H^{s-1}}+\no{\dib_b^*u}^2_{H^{s-1}}\le c\no{\Box_bu}_{H^{s-2}}\|u\|_{H^s}+c_s\no{u}^2_{H^{s-1}}$$
it follows that
\begin{eqnarray}
\label{424b}
\no{u}_{H^s}^2\le c_{s,\eps}\left(\no{\Box_bu}^2_{H^{s-2}}+\no{u}_{H^{s-1}}^2\right)+c\no{T_\epsilon^su}_{L^2}^2. 
\end{eqnarray}

Denote by $\alpha_\eps^{1,0}$ and $\alpha_\eps^{0,1}$ the $(1,0)$-part and $(0,1)$-part of the real $1$-form 
$\alpha_\eps:=-\{\T{Lie}\}_{T_\eps}(\gamma)$. Now we express $\alpha_\eps^{1,0}$ and $\alpha_\eps^{0,1}$  in a local basis. Let $\alpha_{\eps,j}$ be the $T$-component of the commutator $[T_\eps,L_j]$. Then 
 \begin{eqnarray}
 \label{425}
 \begin{split}
\alpha^{1,0}_\eps(L_j)=-\left(\{\T{Lie}\}_{T_\eps}(\gamma)\right)(L_j)=-\left(T_\eps\gamma(L_j)-\gamma([T_\eps,L_j])\right)=\gamma([T_\eps,L_j])=\alpha_{\eps,j}
 \end{split}
 \end{eqnarray}
 and hence 
$\alpha^{1,0}_\eps=\sum_{j=1}^n\alpha_{\eps,j}\om_j$ and $\alpha^{0,1}_\eps=\overline{\alpha^{1,0}_\eps}=\sum_{j=1}^n\overline{\alpha_{\eps,j}}\bom_j$.
 The commutators $[\dib_b,T_\eps]$, $[\dib_b^*,T_\eps]$ and the forms $\alpha^{1,0}_\eps$ and $\alpha^{0,1}_\eps$  are related by 
 $$[\dib_b,T_\eps]=\alpha^{0,1}_\eps \we T+\tilde{\mathcal X}_\eps =b_\eps\alpha^{0,1}_\eps \we T_\eps +{\mathcal X}_\eps ;$$
 $$[\dib_b^*,T_\eps]=-\alpha^{1,0}_\eps \lrcorner T+\tilde{\mathcal Y}_\eps =-b_\eps\alpha^{1,0}_\eps \lrcorner T_\eps +{\mathcal Y}_\eps .$$
 Here, $\mathcal X_\eps:C^\infty_{0,q}(M)\to C^\infty_{0,q+1}(M)$ and $\mathcal Y_\eps:C^\infty_{0,q}(M)\to C^\infty_{0,q-1}(M)$ so that 
 $\no{\mathcal X_\eps u}_{L^2}\le c_\eps\no{\nabla_b u}_{L^2}$ and $\no{\mathcal Y_\eps u}_{L^2}\le c_\eps\no{\nabla_b u}_{L^2}$.  
 In general, for any $s\ge 1$, there exist $\mathcal X_{s,\eps}:C_{q}^\infty(M)\to C_{q+1}^\infty(M)$ and  $\mathcal Y_{s,\eps}:C_{q}^\infty(M)\to C_{q-1}^\infty(M)$  such that
\begin{eqnarray}
\label{comTs}\begin{split}
[\dib_b,T^s_\eps]=&sb_\eps\alpha^{0,1}_\eps\we T^s_\eps+\mathcal X_{s,\eps};\\
[\dib_b^*,T^s_\eps]=&-s b_\eps\alpha^{1,0}_\eps\lrcorner T^s_\eps+\mathcal Y_{s,\eps},
\end{split}
\end{eqnarray}
and $\no{\mathcal X_{s,\eps} u}_{L^2}+\no{\mathcal Y_{s,\eps}u}_{L^2}\le c_{s,\eps}\no{\nabla_b u}_{H^{s-1}}$. Now we are ready to prove \emph{a priori} estimates and the estimates for elliptic regulation. 
\begin{theorem}
	Assume that for any $\epsilon>0$ there exist a vector $T_\epsilon$ and  a constants $c,c_1,c_2,c_\epsilon>0$ such that $c_1\le \gamma(T_\epsilon)\le c_2$ uniformly in $\epsilon$ and 
	\begin{eqnarray}\label{weakcompact}
	\frac{1}{\eps}\no{|\{\T{Lie}\}_{T_\eps}(\gamma)|u}_{L^2}^2+\no{u}_{L^2}^2\le c \left(\no{\dib_b u}_{L^2}^2+\no{\dib_b^*u}_{L^2}^2\right)+c_\epsilon \no{u}^2_{H^{-1}}
	\end{eqnarray}
	holds for all $u\in C^\infty_{0,q}(M)$ with $q_0\le q\le n-q_0$. Then
\begin{eqnarray}
\label{priori}
\no{u}^2_{H^s}&\le &c_s\left(\no{\dib_bu}^2_{H^s}+\no{\dib_b^*u}^2_{H^s}+\no{u}_{L^2}^2\right)\\
\no{\dib_bu}^2_{H^s}+\no{\dib_b^*u}^2_{H^s}&\le &c_s\left(\no{\Box_bu}^2_{H^s}+\no{u}_{L^2}^2\right)\\
\no{\dib_b\dib_b^*u}^2_{H^s}+\no{\dib_b^*\dib_bu}^2_{H^s}&\le& c_s\left(\no{\Box_bu}^2_{H^s}+\no{u}_{L^2}^2\right)
\end{eqnarray}
for all $u\in C^\infty_{0,q}(M)$ and nonnegative $s\in\Z$. Furthermore, for any $s\in \mathbb N$, there exists $ \delta_s>0$ such that 
\begin{eqnarray}\label{eqn:elliptic reg est}
\no{u}_{H^s}^2\le c_s(\no{\Box_b^\delta u}_{H^s}^2+\no{u}^2)
\end{eqnarray}
holds for all $u\in C^\infty_{0,q}(M)$ uniformly in $\delta\in(0,\delta_s)$. The operator $\Box_b^\delta=\Box_b+\delta \left(T^*T+\T{Id}\right)$.
\end{theorem}
\begin{proof} We prove \eqref{priori} by inducting in $s$. It is easy to see that \eqref{priori} holds for $s=0$. 
We now assume that \eqref{priori} holds for $s-1$ with $s\ge 1$ and we are going to prove this estimate still holds on level $s$. We first fix $\eps>0$ independent of $s$. 	
We start with \eqref{weakcompact} with $u$ replaced by $T^s_\eps u$ and use the equality 
$$|\alpha_\eps|^2|u|^2=2\left(|\alpha^{0,1}_\eps\we u|^2+|\alpha^{1,0}_\eps\lrcorner u|^2\right)$$ to see
	\begin{eqnarray}
	\label{t420}
	\begin{split}
\no{T^s_\eps u}_{L^2}^2+&\frac{1}{\eps}\left(\no{\alpha^{0,1}_\eps\we T^s_\eps u}_{L^2}^2+\no{\alpha_\eps^{1,0}\lrcorner T^s_\eps u}_{L^2}^2\right)	\\
\le &c\left( \no{\dib_bT^s_\eps  u}_{L^2}^2+\no{\dib_b T^s_\eps u}_{L^2}^2\right)+c_\eps\no{T^s_\eps  u}_{H^{-1}}^2.
	\end{split}
	\end{eqnarray}	
	From \eqref{comTs},  we have
\begin{eqnarray}
\label{t421a}
\begin{split}
\no{\dib_bT^s_\eps  u}_{L^2}^2\le&\, 3\no{T^s_\eps\dib_bu}_{L^2}^2+cs^2\no{\alpha_\eps^{0,1} \we T^s_\eps u}_{L^2}^2+c_{\eps,s}\no{\nabla_b u}^2_{H^{s-1}},\\\
\no{\dib_b^*T^s_\eps  u}_{L^2}^2\le& \, 3\no{T^s_\eps\dib_b^*u}_{L^2}^2+cs^2\no{\alpha^{1,0}_\eps\lrcorner T^s_\eps u}_{L^2}^2+c_{\eps,s}\no{\nabla_b u}^2_{H^{s-1}},
\end{split}
\end{eqnarray}	
where we have used that $b_\eps$ is uniformly bounded. 
However, for each $s$, there exists $\eps_s$ such that for any $\eps<\eps_s$ the expression  
$s^2\left(\no{\alpha^{0,1}_\eps \we T^s_\eps u}_{L^2}^2+\no{\alpha^{1,0}_\eps \lrcorner T^s_\eps u}_{L^2}^2\right)$ can be absorbed. Thus 
	\begin{eqnarray}
	\label{t422a}
	\begin{split}
\no{T^s_\eps u}_{L^2}^2+&\frac{1}{\eps}\left(\no{\alpha^{0,1}_\eps \we T^s_\eps u}_{L^2}^2+\no{\alpha^{1,0}_\eps \lrcorner T^s_\eps u}_{L^2}^2\right)\\
\le c&\left( \no{T^s_{\eps}\dib_b u}_{L^2}^2+\no{T^s_{\eps}\dib^*_b u}_{L^2}^2\right)+c_{s,\eps}\left(\no{\nabla_b u}^2_{H^{s-1}}+\no{u}_{H^{s-1}}^2\right).
	\end{split}
	\end{eqnarray}	
	Combining with \eqref{421} and \eqref{424}, it follows
	\begin{eqnarray}
	\label{t422}
	\begin{split}
	\no{u}_{H^s}^2
	\le c\left( \no{T^s_{\eps}\dib_b u}_{L^2}^2+\no{T^s_{\eps}\dib^*_b u}_{L^2}^2\right)+c_{s,\eps}\left(\no{\dib_b u}^2_{H^{s-1}}+\no{\dib_b^* u}^2_{H^{s-1}}+\no{u}_{H^{s-1}}^2\right)\\
	\end{split}
	\end{eqnarray}	
	Thus the first \emph{a priori} estimate follows by using the inductive hypothesis for $\no{u}^2_{H^{s-1}}$.\\
	
	For the second \emph{a priori} estimate, it follows from \eqref{comTs} that
	\begin{eqnarray}
	\label{t424a}
	\begin{split}
	\no{T^s_\eps \dib_b u}_{L^2}^2=&(T^s_\eps \dib_b^*\dib_bu,T^s_\eps u)+([\dib_b^*,T^s]\dib_bu,T^s_\eps u)+(T^s_\eps\dib_b u, [T^s_\eps,\dib_b]u)\\
	=&(T^s_\eps \dib_b^*\dib_bu,T^s_\eps u)-s(b_\eps\alpha^{1,0}_\eps\lrcorner T^s_\eps\dib_bu,T^s_\eps u)+(\mathcal Y_{s,\eps}\dib_bu,T^s_\eps u)\\
	&+s(T^s_\eps\dib_b u, b_\eps\alpha^{0,1}_\eps\we T^s_\eps u)+(T^s_\eps\dib_b u, \mathcal X_{s,\eps} u);
	\end{split}
	\end{eqnarray}	
	and similarly,
	\begin{eqnarray}
	\label{t424b}
	\begin{split}
	\no{T^s_\eps \dib_b^* u}_{L^2}^2
	=&(T^s_\eps \dib_b\dib_b^*u,T^s_\eps u)+s(b_\eps\alpha_\eps^{0,1}\we T^s_\eps\dib_b^*u,T^s_\eps u)+(\mathcal X_{s,\eps}\dib_b^*u,T^s_\eps u)\\
	&-s(T^s_\eps\dib_b^* u, b_\eps\alpha_\eps^{1,0}\lrcorner T^s_\eps u)+(T^s_\eps\dib_b^* u, \mathcal Y_{s,\eps} u),
	\end{split}
	\end{eqnarray}	

Next, we sum and use the equality $(\alpha^{0,1}_\eps \we u,v)=( u,\alpha^{1,0}_\eps\lrcorner v)$, the $(sc)-(lc)$ inequality, and the uniform boundedness of $b_\eps$ to obtain 
 	\begin{eqnarray}
 	\label{t424c}
 	\begin{split}
 	&\no{T^s_\eps \dib_b u}_{L^2}^2+	\no{T^s_\eps \dib_b^* u}_{L^2}^2\le (T^s_\eps \Box_bu,T^s_\eps u)\\
 	&+c_{s,\eps}\left(\no{\nabla_b\dib_bu}^2_{H^{s-1}}+\no{\nabla_b\dib_b^* u}^2_{H^{s-1}}+\no{\nabla_bu}^2_{H^{s-1}}+\no{\dib_bu}^2_{H^{s-1}}+\no{\dib_b^* u}^2_{H^{s-1}}+\no{u}^2_{H^{s-1}}\right)\\
 	&+\T{sc}\left(\no{T^s_\eps \dib_bu}_{L^2}^2+\no{T^s_\eps\dib_b^* u}_{L^2}^2+\no{T^s_\eps u}_{L^2}^2\right)+\T{lc}s^2\left(\no{\alpha_\eps^{0,1}\we T^s_\eps u}_{L^2}^2+\no{\alpha_\eps^{1,0}\lrcorner T^s_\eps u}_{L^2}^2\right)\\
 	\end{split}
 	\end{eqnarray}	
By \eqref{t422a}, we may absorb the term the last line by choosing $\eps<\eps_s$ sufficiently small. We can bound $\|u\|_s^2$ with (\ref{t422}) \eqref{t424c} and observe
\begin{eqnarray}
\label{t423}
\begin{split}
\no{u}^2_{H^s}+\no{\dib_b u}_{H^s}^2+\no{\dib_b^*  u}_{H^s}^2
\le &\ c(T^s_\eps \Box_b u,T^s_\eps u )+I
\end{split}
\end{eqnarray}	
where 
\begin{eqnarray}
\label{t423b}
\begin{split}I=&c_{s,\eps}\left(\no{\nabla_b\dib_bu}^2_{H^{s-1}}+\no{\nabla_b\dib_b^* u}^2_{H^{s-1}}+\no{\nabla_bu}^2_{H^{s-1}}+\no{\dib_bu}^2_{H^{s-1}}+\no{\dib_b^* u}^2_{H^{s-1}}+\no{u}^2_{H^{s-1}}\right)\\
\le& c_{s,\eps}\Big(\no{\dib_bu}_{s}\no{\dib_bu}_{H^{s-1}}+\no{\dib_b^*u}_{s}\no{\dib_b^*u}_{H^{s-1}}+\no{u}_{H^s}\no{u}_{H^{s-1}}\\
& +\no{\dib_b\dib_b^*u}^2_{H^{s-1}}+\no{\dib_b^*\dib_bu}^2_{H^{s-1}}+\no{\dib_bu}^2_{H^{s-1}}+\no{\dib_b^* u}^2_{H^{s-1}}+\no{u}^2_{H^{s-1}}\Big)\\
\le& \T{sc}\left(\no{u}^2_{H^s}+\no{\dib_b u}_{H^s}^2+\no{\dib_b^*  u}_{H^s}^2\right)+c_{\eps,s}\T{lc}\left(\no{\Box_bu}^2_{H^{s-1}}+\no{u}_{L^2}^2\right)
\end{split}
\end{eqnarray}	
here the second inequality follows by the middle line of \eqref{421} and the last inequality follows by the inductive hypothesis. Thus, \begin{eqnarray}
\label{t42c}
\begin{split}
\no{u}^2_{H^s}+\no{\dib_b u}_{H^s}^2+\no{\dib_b^*  u}_{H^s}^2
\le &\ c(T^s_\eps \Box_b u,T^s_\eps u )+c_{\eps,s}\left(\no{\Box_bu}^2_{H^{s-1}}+\no{u}_{L^2}^2\right)
\end{split}
\end{eqnarray}
The second \emph{a priori} estimate now follows immediately by \eqref{t42c}. The last \emph{a priori} estimate follows by the second \emph{a priori} estimate and the inequality    
\begin{eqnarray}
\nonumber\label{t424}
\begin{split}
&\no{\dib_b\dib_b^*u}^2_{H^s}+\no{\dib_b\dib_b^*u}^2_{H^s}=\no{\Box_bu}^2_{H^s}-2\Re(\Lambda^s \dib_b\dib_b^*u,\Lambda^s \dib_b^*\dib_bu)\\
&= \no{\Box_bu}_{{H^s}}^2-2\Re\Big((\Lambda^s \dib_b\dib_b\dib_b^*u,\Lambda^s \dib_bu)+([\dib_b,\Lambda^s]\dib_b\dib_b^*u,\Lambda^s \dib_bu)+(\Lambda^s \dib_b\dib_b^*u,[\Lambda^s, \dib_b^*]\dib_bu)\Big)\\
&\le\no{\Box_bu}^2_{{H^s}}+\T{sc}\no{\dib_b\dib_b^*u}^2_{H^s}+\T{lc}\no{\dib_bu}^2_{H^s}.
\end{split}
\end{eqnarray}

We now prove (\ref{eqn:elliptic reg est}), the estimate that allows us to use the method of  elliptic regularization. We first show 
\begin{equation}
\label{t45}\no{u}^2_{H^{s+1}}\le c(T^s_\eps T^* T u,T^s_\eps u)+c_{\eps,s}\left(\no{\Box_bu}^2_{H^{s-1}}+\no{u}^2_{H^s}\right).
\end{equation}
This estimate follows quickly by combining \eqref{424b}, 
$$\no{u}^2_{H^{s+1}}\le c\no{TT^{s}_\eps u}_{L^2}^2+c_{\eps,s}\left(\no{\Box_bu}^2_{H^{s-1}}+\no{u}^2_{H^s}\right),$$
and 
\begin{eqnarray}
\begin{split}\no{TT^{s}_\eps u}^2
=&(T^s_\eps T^* T u,T^s_\eps u)+([T^*T, T^{s}_\eps]  u,T^s_\eps u))\\
\le&(T^s_\eps T^*Tu,T^s_\eps u)+\T{sc}\no{u}_{H^{s+1}}^2+c_{\eps,s}\T{lc}\no{u}^2_{H^s}.
\end{split}
\end{eqnarray}
But \eqref{t45} and \eqref{t42c}, we get 
\begin{eqnarray}
\label{t45b}
\begin{split}
\no{u}^2_{H^s}+&\no{\dib_b u}_{H^s}^2+\no{\dib_b^*  u}_{H^s}^2+\delta \no{u}^2_{H^{s+1}}
\\
\le&\, c(T^s_\eps \left((\Box_b+\delta (T^*_{\eps}T_\eps+I)) u,T^s_\eps u \right)
+c_{\eps,s}\left(\no{\Box_bu}^2_{H^{s-1}}+\delta \no{u}^2_{H^s}+\no{u}_{L^2}^2\right)\\
\le &\, c_{\eps,s}\left(\no{\Box_b^\delta u}^2_{H^s}+\no{\Box_b^\delta u^2}_{H^{s-1}}+\no{u}_{L^2}^2\right)\\ 
&\, +\delta c_{\eps,s}\no{u}^2_{H^s}+\delta^2c_{\eps,s}\no{u}^2_{H^{s+1}}, 
\end{split}
\end{eqnarray}
where we have used that $\no{\Box_bu}^2_{H^{s-1}}\le 2\no{\Box_b^\delta u}^2_{H^{s-1}}+\delta^2\no{u}^2_{H^{s+1}}$. Now we fix $\eps$ depending on $s$ so the above estimates hold, and choose $\delta_s$ 
such that $\delta c_{\eps,s}$ and $\delta^2 c_{\eps,s}$ are small for any $\delta\le \delta_s$. Thus, we may absorb the last line and obtain an inequality stronger than the desired inequality. 
\end{proof}
\subsection{Proof of Theorem \ref{thm:global regularity}}
{\it (i) Proof of the global regularity of $G_q$.} 
 For a given $\varphi\in C^\infty_{0,q}(M)$, we first prove that $G_q\varphi\in C^\infty_{0,q}(M)$  by elliptic regularization using the elliptic perturbation  $\Box_b^\delta:=\Box_b+\delta(T^*T+\T{Id})$.
 First, though, we make one short remark. Since $\mathcal H_{0,q}(M)$ is finite dimensional, and all norms on finite dimensional vector spaces are equivalent, it follows that $\|u\|_{L^2} \approx \|u\|_{H^{-1}(M)}$ for any
 $u\in\mathcal H_{0,q}(M)$. From this equivalence
 and the density of smooth forms, we may conclude that harmonic forms are smooth and $\|u\|_{H^s} \approx \|u\|_{L^2}$ where the equivalence depends on $s$ but is independent of the harmonic $(0,q)$-form $u$.

Let $Q^{\delta}_b(\cdot,\cdot)$ be the quadratic form on $H^1_{0,q}(M)$ defined by 
$$Q^{\delta}_b(u,v)=Q_b(u,v)+\delta((Tu,Tv)+(u,v))=(\Box_b^\delta u,v)$$
By (\ref{424}), we have $\no{u}^2_1\le c_\delta Q_b^\delta(u,u)$ for any $u\in H^1_{0,q}(M)$. Consequently, $\Box_b^\delta$ is a self-adjoint, elliptic operator  with inverse $G_q^\delta$.  By elliptic theory, 
we know that if $\varphi\in C^\infty_{0,q}(M)$, then $G_q^\delta \varphi\in C^\infty_{0,q}(M)$. We can therefore use (\ref{eqn:elliptic reg est}) with $u=G^\delta_q\varphi$ and estimate
$$\no{G^\delta_q\varphi}^2_{H^s}\le c_s\left( \no{\Box_b^\delta G_q^\delta\varphi}^2_{H^s}+\no{G_q^\delta u}_{L^2}^2\right)=c_s\left( \no{\varphi}^2_{H^s}+\no{G_q^\delta \varphi}_{L^2}^2\right)$$
where the equality follows from the identity $\Box_b^\delta G_q^\delta=Id$ (since $\T{Ker}(\Box_b^\delta)=\{0\}$). 
By Lemma \ref{lem:lem below}, 
$\no{G_q^\delta \varphi}_{L^2}\le c\no{\varphi}_{L^2}$ uniformly in $\delta$ when $1 \leq q_0\le q\le n-q_0$. 
Thus, $\no{G_q^\delta \varphi}_{H^s}$ is uniformly bounded and hence there exists a subsequence $\delta_k$ and $\tilde u\in H^s_{0,q}(M)$ such that $G^{\delta_k}_q \varphi\to \tilde u$ weakly in $H^s_{0,q}(M)$. 
Consequently, $G^{\delta_k}_q \varphi\to \tilde u$ weakly in the $Q_b$-norm, which means that if $v\in H^2_{0,q}(M)$, then
$$\lim_{\delta_k\to 0}Q_b(G^{\delta_k}_q\varphi,v)=Q_b(\tilde u,v).$$
On the other hand, 
$$Q_b(G_q\varphi,v)=(\varphi,v)=Q^\delta_b(G^\delta_q \varphi,v)=Q_b(G^\delta_q \varphi,v)+\delta\left((TG^\delta_q \varphi,Tv)+(G^\delta_q \varphi,v)\right)$$
for all $v\in H^2_{0,q}(M)$. It follows that
$$|Q_b((G^\delta_q \varphi-G_q\varphi),v)|\le \delta \no{G^\delta_q\varphi}_{L^2}\no{v}_{2}\le c\delta \no{\varphi}_{L^2}\no{v}_2\to 0\quad \T{as $\delta\to 0$}$$
where we have again used the inequality  $\no{G_q^\delta \varphi}_{L^2}\le c\no{\varphi}_{L^2}$ uniformly in $\delta$. 
We therefore have $G_q\varphi=\tilde u\in H^s_{0,q}(M)$. This holds for arbitrary $s\in \mathbb N$, so the Sobolev Lemma implies that $G_q\varphi\in C^\infty_{0,q}(M)$.\\

{\it (ii) Proof of the exact regularity of $G_q$, $\dib_bG_q$, $\dib_b^*G_q$, $I-\dib_b^*\dib_bG_q$ and $I-\dib_b\dib_b^*G_q$.} 
For $\varphi\in C^\infty_{0,q}(M)$, we  use the estimates in Theorem \ref{comTs} with $u=G_q\varphi\in C^\infty_{0,q}(M)$ and observe
\begin{eqnarray}
\label{Cinfty}
\begin{split}
\no{G_q\varphi}_{H^s}^2+&\no{\dib_bG_q\varphi}_{H^s}^2+\no{\dib_b^*G_q\varphi}_{H^s}^2+\no{\dib_b^*\dib_bG_q\varphi}_{H^s}^2+\no{\dib_b\dib_b^*G_q\varphi}_{H^s}^2\\
\le &c_s(\no{\Box_bG_q\varphi}^2_{H^s}+\no{G_q\varphi}_{L^2}^2)=c_s(\no{(I-H_q)\varphi}^2_{H^s}+\no{G_q\varphi}_{L^2}^2) \\
\le&c_s(\no{\varphi}^2_{H^s}+\no{H_q\varphi}^2_{H^s}+\no{G_q\varphi}_{L^2}^2)\\
\le& c_s(\no{\varphi}^2_{H^s}+\no{\varphi}_{L^2})\le c_s\no{\varphi}_{H^s}^2.
\end{split}
\end{eqnarray} 

We have shown the Sobolev estimate  holds for  $\varphi\in C^\infty_{0,q}(M)$, and this space is dense in $H^s_{0,q}(M)$. Consequently, since $G_q$ is continuous on $L^2_{0,q}(M)$, so that the Sobolev estimate carries over to $\varphi\in H^s_{0,q}(M)$. 
This means $G_q$, $\dib_bG_q$, $\dib_b^*G_q$, $I-\dib_b^*\dib_bG_q$, and $I-\dib_b\dib_b^*G_q$ are exactly regular. \\  

For $G_q\dib_b$ and $I-\dib_b^*G_q\dib_b$, let $\varphi\in C^\infty_{0,q-1}(M)$. By using the estimate \eqref{t42c} with $u=G_q\dib_b\varphi\in C^\infty_{0,q}(M)$ we obtain
\begin{eqnarray}
\begin{split}
	\no{G_q\dib_b\varphi}_{H^s}^2+\no{\dib_b^*G_q\dib_b\varphi}_{H^s}^2\le &c(T^s_\eps \Box_b G_q\dib_b \varphi, T^s_\eps G_q\dib_b\varphi)+c_{\eps,s}\left(\no{\Box_bG_q\dib_b\varphi}^2_{H^{s-1}}+\no{G_q\dib_b\varphi}^2_{L^2}\right)\\
	\le &c(T^s_\eps(I-H_q) \dib_b \varphi, T^s_\eps G_q\dib_b\varphi)+c_{\eps,s}\left(\no{(I-H_q)\dib_b\varphi}^2_{H^{s-1}}+\no{\varphi}^2_{L^2}\right)\\
	\le &c\Big((T^s_\eps \varphi, T^s_\eps \dib_b^*G_q\dib_b\varphi)+([T^s_\eps ,\dib_b]\varphi,  T^s_\eps G_q\dib_b\varphi)+(T^s_\eps\varphi,[\dib_b^*,T^s_\eps] G_q\dib_b\varphi)\\
	&+(T^*_\eps T^{s}_\eps H_q\varphi, T^{s-1}_\eps G_q\dib_b\varphi) \Big)+c_{\eps,s}\left(\no{\dib_b\varphi}^2_{H^{s-1}}+\no{\varphi}^2_{L^2}\right)\\
	\le &\T{sc}\left(\no{G_q\dib_b\varphi}_{H^s}^2+\no{\dib_b^*G_q\dib_b\varphi}_{H^s}^2\right)+c_\eps\T{lc}\left(\no{\varphi}_{H^s}^2+\no{H_q\varphi}_{H^{s+1}}^2\right)
\end{split}
\end{eqnarray} 
Absorbing the  sc term by the LHS and using the fact that $\no{H_q\dib_b\varphi}_{H^{s+1}}\le c\no{\dib_b\varphi}_{L^2}\le c\no{\varphi}_1$, we conclude that 
$$	\no{G_q\dib_b\varphi}_{H^s}^2+\no{(I-\dib_b^*G_q\dib_b)\varphi}_{H^s}^2\le c_s\no{\varphi}_{H^s}^2$$
for all $\varphi\in C^\infty_{0,q-1}(M)$. As in the above argument, this Sobolev estimate also holds for $\varphi\in H^s_{0,q-1}(M)$. 
We may then prove the exact regularity of $G_q\dib_b^*$ and $(I-\dib_bG_q\dib_b^*)$ for forms of degree $(0,q+1)$ similarly.\\

Finally, exact regularity of $\dib_b^*G_q$, $G_q\dib_b$, $\dib_bG_q$, $G_q\dib_b^*$ implies  that $\dib_b^*G_q^2\dib_b$ and $\dib_bG_q^2\dib_b^*$ are also exactly regular. 
It is known that on the top degrees the Green operators $G_0$ and $G_n$ are given by $\dib_b^*G_1^2\dib_b$ and $\dib_bG_{n-1}^2\dib_b^*$, respectively.
Moreover, $\dib_bG_0=G_1\dib_b$, $\dib_b^*G_n=G_{n-1}\dib_b^*$. Therefore, if $q=1$ then $G_0, \dib_bG_0$, $G_n$, $\dib_b^*G_n$ are exactly regular.\\

The proof of Theorem ~\ref{thm:global regularity} is complete, pending the following technical lemma.
\begin{lemma}\label{lem:lem below} Fix $1\le q\le n-1$.
Let $M^{2n+1}$ be an abstract CR manifold that the $L^2$ basic estimate 
	\begin{equation}
	\label{L2basic1} \no{u}_{L^2}^2\le c(\no{\dib_bu}_{L^2}^2+\no{\dib_b^*u}_{L^2}^2) + C \|u\|_{H^{-1}}^2
	\end{equation}
	holds for all $u\in \T{Dom}(\dib_b)\cap\T{Dom}(\dib_b)$.	
	Then $\no{G_q^\delta \varphi}\le c \no{\varphi}$ uniformly in $\delta$ for $\varphi\in L^2_{0,q}(M)$.
	\end{lemma}
\begin{proof}Fix $\delta>0$.  It suffices to show that 
	\begin{equation}\label{n1} \no{u}_{L^2}^2\le c Q_b^\delta(u,u)\end{equation} 
	holds for some constant $c>0$ that is independent of $\delta>0$ and $u\in H^1_{0,q}(M)$. 
	The basic estimate certainly implies that
$$\no{u}_{L^2}^2\le cQ_b^\delta(u,u) +c'\no{u}^2_{H^{-1}},$$
uniformly in $\delta$ for all $u\in H^1_{0,q}(M)$. Assume that \eqref{n1} fails. Then there exists $u_k$ with $\no{u_k}_{L^2}^2=1$ so that 
	\begin{equation}\label{n2} \no{u_k}_{L^2}^2\ge k Q_b^\delta(u_k,u_k).\end{equation} 
For $k$ sufficiently large, we can use \eqref{n2} and absorb $Q_b^\delta(u_k,u_k)$ by $\no{u_k}^2$ and to prove 
	\begin{equation}\label{n3} 
\no{u_k}_{L^2}^2\le 2c'\no{u_k}^2_{H^{-1}}.\end{equation} 
Since $L^2_{0,q}(M)$ is compact in $H^{-1}_{0,q}(M)$, there exists a subsequence $u_{k_j}$ that converges in $H^{-1}_{0,q}(M)$. Thus,  \eqref{n3} forces $u_{k_j}$ to converge in $L^2_{0,q}(M)$; and \eqref{n2} forces $u_{k_j}$ to converge in the $Q_b^\delta(\cdot,\cdot)$-norm as well. The limit $u$ satisfies $\no{u}_{L^2}=1$. However, a consequence of \eqref{n2} is that $\no{u}_{L^2}=0$ since 
$Q^\delta_b(u,u)\ge \delta \no{u}_{L^2}^2$.  This is a contradiction and \eqref{n1} holds.
	\end{proof}

\subsection{Proof of Theorem \ref{thm:CRpsh (0,1)}}
   \begin{proof}[Proof of Theorem \ref{thm:CRpsh (0,1)}] Assume that there exists a global contact form $\tilde \gamma$ and a smooth function $\tilde{c}_{00}$ such that the extended Levi matrix $\mathcal{\tilde M}:=\{\tilde c_{ij}\}_{i,j=0}^n$ is  positive semidefinite.  
    Thus we can use the Schwarz inequality for the two vectors $u=(0,u_1,\dots, u_n)$, $v=(1,0,\dots,0)$ in $\C^{n+1}$ and get 
    	\begin{eqnarray}
    	\label{Schwarz}\left|\sum_{j=1}^n\tilde c_{0j}u_j\right|^2=\left|\mathcal{\tilde M}(u,v)\right|^2\le \mathcal{\tilde M}(u,u)\mathcal{\tilde M}(v,v)=\left|\sum_{i,j=1}^n\tilde c_{ij}u_i\bar u_j\right||\tilde{c}_{00}|.
    		\end{eqnarray}
    		On the other hand,  there exists a smooth function $h$ in $M$ such that $\tilde \gamma=e^{-h}\gamma$. Thus, 
    		\begin{eqnarray}
    		\label{dtildegamma}\begin{split}
    		d\tilde\gamma=&\, e^{-h}\,d\gamma-e^{-h}\,dh\we \gamma\\
    		=&-e^{-h}\left(\sum_{i,j=1}^nc_{ij}\,\om_i\we\bom_j+\sum_{j=1}^nc_{0j}\,\gamma\we \om_j
		+\sum_{j=1}^n\overline{c_{0j}}\,\gamma\we \bom_j\right)+e^{-h}\sum_{j=1}^n\left(L_j(h)\,\gamma\we \om_j+\bar L_j(h)\,\gamma\we \bom_j\right)\\
    		=&-\sum_{i,j=1}^ne^{-h}c_{ij}\,\om_i\we\bom_j-\sum_{j=1}^n\left(e^{-h}(c_{0j}-L_j(h))\,\gamma\we \om_j+e^{-h}\overline{(c_{0j}-L_j(h))}\,\gamma\we \bom_j\right)\\
    		=&-\left(\sum_{i,j=1}^n\tilde c_{ij}\,\om_i\we \bom_j+\sum_{j=1}^n\tilde c_{0j}\,\tilde \gamma\we \om_j+\sum_{j=1}^n\overline{ \tilde c_{0j}}\,\tilde \gamma\we \bom_j \right),
    		\end{split}
    		\end{eqnarray}
    		where 
    		$\tilde{c}_{ij}=e^{-h}c_{ij}$ for $i,j=1,\dots,n$ and $\tilde c_{0j}=c_{0j}-L_j(h)$. Substituting these $\tilde c_{ij}$ into \eqref{Schwarz}, we get 
    		\begin{eqnarray}\label{new1}\left|\sum_{j=1}^n(c_{0j}-L_j(h))u_j\right|^2\le  |\tilde c_{00}|\left|\sum_{i,j=1}^n \tilde c_{ij}u_i\bar u_j\right|  = |\tilde c_{00}|e^{-h} \sum_{i,j=1}^n c_{ij}u_i\bar u_j	\end{eqnarray}
    		holds at any $x\in M$ and any vector $u=(u_1,\dots,u_n)\in \C^n$. 
Recall that
    $$\alpha=-\{\T{Lie}\}_{T}(\gamma)=-\sum_{j=1}^n\left(d\gamma(T,L_j)\om_j+d\gamma(T,\bar L_j)\bom_j\right)=\sum_{j=1}^n\left(c_{0j}\om_j+\bar c_{0j}\bom_j\right),$$
    (where the last inequality follows by \eqref{dgamma}).
     If $L=\sum_ju_jL_j\in T^{1,0}M$ then we rewrite  \eqref{new1} as
    	$$|\alpha(L)-dh(L)|\le c d\gamma(L\we\bar L).$$ 
    This calculation implies that $\alpha$ is exact on the null space of Levi form. 
    The argument from \cite[Proposition 1]{StZe15} then finishes the proof. Straube and Zeytuncu's work shows that if $\alpha$ is exact on the null space of the Levi form then for any $\eps$, 
    then we can find a vector $T_\eps$ traversal to $T^{0,1}M\oplus T^{0,1}M$ such that the $T$-component of $[T_\eps,L]$ is less than $\epsilon$ for any unit vector field $L\in T^{1,0}$. 
    Their result is stated for embedding manifolds in $\C^N$ but a careful examination of the proof reveals that embeddedness is an unnecessary assumption with their proof.
        \end{proof}

\bibliographystyle{alpha}

\begin{thebibliography}{KPZ12}
	
	\bibitem[Bog91]{Bog91}
	A.~Boggess.
	\newblock {\em CR Manifolds and the Tangential Cauchy-Riemann Complex}.
	\newblock Studies in Advanced Mathematics. CRC Press, Boca Raton, Florida,
	1991.
	
	\bibitem[BPZ15]{BaPiZa15}
	L.~Baracco, S.~Pinton, and G.~Zampieri.
	\newblock Hypoellipticity of the {K}ohn-{L}aplacian {$\square_b$} and of the
	{$\overline\partial$}-{N}eumann problem by means of subelliptic multipliers.
	\newblock {\em Math.~Ann.}, 362(3-4):887--901, 2015.
	\newblock http://dx.doi.org/10.1007/s00208-014-1144-1.
	
	\bibitem[BS86]{BoSh86}
	H.\ Boas and M.-C.\ Shaw.
	\newblock Sobolev estimates for the {L}ewy operator on weakly pseudoconvex
	boundaries.
	\newblock {\em Math.\ Ann.}, 274:221--231, 1986.
	
	\bibitem[BS91]{BoSt91}
	H.\ Boas and E.\ Straube.
	\newblock Sobolev estimates for the complex {G}reen operator on a class of
	weakly pseudoconvex boundaries.
	\newblock {\em Comm.\ Partial Differential Equations}, 16:1573--1582, 1991.
	
	\bibitem[BS93]{BoSt93}
	H.~Boas and E.~Straube.
	\newblock de {R}ham cohomology of manifolds containing the points of infinite
	type, and {S}obolev estimates for the {$\overline\partial$}-{N}eumann
	problem.
	\newblock {\em J.~Geom.~Anal.}, 3(3):225--235, 1993.
	
	\bibitem[BS99]{BoSt99}
	H.\ Boas and E.\ Straube.
	\newblock Global regularity of the $\overline\partial$-{N}eumann problem: a
	survey of the ${L}\sp 2$-{S}obolev theory.
	\newblock In {\em Several {C}omplex {V}ariables ({B}erkeley, {CA},
		1995--1996)}, Mat.\ Sci.\ Res.\ Inst.\ Publ., 37, pages 79--111. Cambridge
	Univ.\ Press, Cambridge, 1999.
	
	\bibitem[CS01]{ChSh01}
	S.-C.\ Chen and M.-C.\ Shaw.
	\newblock {\em Partial Differential Equations in Several Complex Variables},
	volume~19 of {\em Studies in Advanced Mathematics}.
	\newblock American Mathematical Society, 2001.
	
	\bibitem[Gra63]{Gra63}
	H.~Grauert.
	\newblock Bemerkenswerte pseudokonvexe {M}annigfaltigkeiten.
	\newblock {\em Math. Z.}, 81:377--391, 1963.
	
	\bibitem[Har11]{Har11}
	P.~Harrington.
	\newblock Global regularity for the {$\overline\partial$}-{N}eumann operator
	and bounded plurisubharmonic exhaustion functions.
	\newblock {\em Adv.\ Math.}, 228(4):2522--2551, 2011.
	\newblock http://dx.doi.org/10.1016/j.aim.2011.07.008.
	
	\bibitem[H{\"o}r65]{Hor65}
	L.~H{\"o}rmander.
	\newblock ${L}^{2}$ estimates and existence theorems for the $\bar \partial $
	operator.
	\newblock {\em Acta Math.}, 113:89--152, 1965.
	
	\bibitem[HR11]{HaRa11}
	P.~Harrington and A.~Raich.
	\newblock Regularity results for $\bar\partial_b$ on {CR}-manifolds of
	hypersurface type.
	\newblock {\em Comm.\ Partial Differential Equations}, 36(1):134--161, 2011.
	
	\bibitem[HR15]{HaRa15}
	P.~Harrington and A.~Raich.
	\newblock Closed range for $\bar\partial$ and $\bar\partial_b$ on bounded
	hypersurfaces in {S}tein manifolds.
	\newblock {\em Ann. Inst. Fourier (Grenoble)}, 65(4):1711--1754, 2015.
	
	\bibitem[Kha]{Kha18}
	T.V. Khanh.
	\newblock The Kohn-Laplace equation on abstract cr manifolds: local regularity.
	\newblock {\em submitted}.
	
	\bibitem[Kha16]{Kha16b}
	T.V. Khanh.
	\newblock Equivalence of estimates on a domain and its boundary.
	\newblock {\em Vietnam J.\ Math}, 44(1):29--48, 2016.
	
	\bibitem[Koh85]{Koh85}
	J.J. Kohn.
	\newblock Estimates for $\bar\partial_b$ on compact pseudoconvex {CR}
	manifolds.
	\newblock In {\em Proceedings of Symposia in Pure Mathematics}, volume~43,
	pages 207--217. American Mathematical Society, 1985.
	
	\bibitem[Koh86]{Koh86}
	J.J.\ Kohn.
	\newblock The range of the tangential {C}auchy-{R}iemann operator.
	\newblock {\em Duke Math.\ J.}, 53:525--545, 1986.
	
	\bibitem[Koh00]{Koh00}
	J.~J.\ Kohn.
	\newblock Hypoellipticity at points of infinite type.
	\newblock In {\em Analysis, geometry, number theory: the mathematics of {L}eon
		{E}hrenpreis ({P}hiladelphia, {PA}, 1998)}, volume 251 of {\em Contemp.
		Math.}, pages 393--398. Amer. Math. Soc., Providence, RI, 2000.
	
	\bibitem[Koh02]{Koh02}
	J.J. Kohn.
	\newblock Superlogarithmic estimates on pseudoconvex domains and {CR}
	manifolds.
	\newblock {\em Ann.\ of Math.}, 156:213--248, 2002.
	
	\bibitem[KPZ12]{KhPiZa12a}
	T.V. Khanh, S.~Pinton, and G.~Zampieri.
	\newblock Compactness estimates for {$\square\sb {b}$} on a {CR} manifold.
	\newblock {\em Proc. Amer. Math. Soc.}, 140(9):3229--3236, 2012.
	
	\bibitem[KZ11]{KhZa11}
	T.V. Khanh and G.~Zampieri.
	\newblock Estimates for regularity of the tangential
	{$\overline{\partial}$}-system.
	\newblock {\em Math. Nachr.}, 284(17-18):2212--2224, 2011.
	
	\bibitem[Nic06]{Nic06}
	A.~Nicoara.
	\newblock Global regularity for $\bar\partial_b$ on weakly pseudoconvex {CR}
	manifolds.
	\newblock {\em Adv. Math.}, 199:356--447, 2006.
	
	\bibitem[Rai10]{Rai10}
	A.~Raich.
	\newblock Compactness of the complex {G}reen operator on {CR}-manifolds of
	hypersurface type.
	\newblock {\em Math.~Ann.}, 348(1):81--117, 2010.
	
	\bibitem[RS08]{RaSt08}
	A.\ Raich and E.\ Straube.
	\newblock Compactness of the complex {G}reen operator.
	\newblock {\em Math.\ Res.\ Lett.}, 15(4):761--778, 2008.
	
	\bibitem[Sha85]{Sha85}
	M.-C.\ Shaw.
	\newblock ${L}\sp 2$-estimates and existence theorems for the tangential
	{C}auchy-{R}iemann complex.
	\newblock {\em Invent.\ Math.}, 82:133--150, 1985.
	
	\bibitem[Str08]{Str08}
	E.~Straube.
	\newblock A sufficient condition for global regularity of the
	$\overline\partial$-{N}eumann operator.
	\newblock {\em Adv.\ Math.}, 217(3):1072--1095, 2008.
	
	\bibitem[Str10]{Str10}
	E.\ Straube.
	\newblock {\em Lectures on the ${\mathcal{L}}^2$-Sobolev Theory of the
		$\bar\partial$-Neumann Problem}.
	\newblock ESI Lectures in Mathematics and Physics. European Mathematical
	Society (EMS), Z{\"u}rich, 2010.
	
	\bibitem[Str12]{Str12}
	E.~J. Straube.
	\newblock The complex {Green} operator on {CR}-submanifolds of $\mathbb{C}^n$
	of hypersurface type: compactness.
	\newblock {\em Trans. Amer. Math. Soc.}, 364(8):4107--4125, 2012.
	
	\bibitem[SZ15]{StZe15}
	E.~Straube and Y.~Zeytuncu.
	\newblock Sobolev estimates for the complex {G}reen operator on {CR}
	submanifolds of hypersurface type.
	\newblock {\em Invent. Math.}, 201(3):1073--1095, 2015.
	
\end{thebibliography}

\end{document}